\documentclass[12pt]{amsart}
\usepackage[utf8]{inputenc}

\usepackage[boxsize=1em]{ytableau}
\usepackage{amsmath,amssymb,amsthm,mathtools,mathdots,nicefrac,tikz-cd,bbding,stmaryrd,float,multicol,caption,bbm, upgreek, enumitem,semantex,comment}
\usepackage[mathscr]{euscript}
\usepackage[margin=1in, marginparwidth=1.75cm]{geometry}
\usepackage{color}
\usepackage[all]{xy}

\NewSymbolClass\MyBinaryOperator[
    define keys={
        {Lder}{command=\overset{\mathrm{L}}},
        {Rder}{upper=R},
    },
]
\NewObject\MyBinaryOperator\tensor{\otimes}[
    define keys={
        {der}{Lder},
    },
]
\NewObject\MyBinaryOperator\fibre{\times}[
    define keys={
        {der}{Rder},
    },
]

\usepackage[shortalphabetic]{amsrefs}
\usepackage{hyperref}
\hypersetup{colorlinks=true,citecolor=purple,linkcolor=purple}

\theoremstyle{definition}
\newtheorem{theorem}{Theorem}[section]

\theoremstyle{definition}
\newtheorem{corollary}[theorem]{Corollary}

\theoremstyle{definition}
\newtheorem{proposition}[theorem]{Proposition}

\theoremstyle{definition}
\newtheorem{lemma}[theorem]{Lemma}

\theoremstyle{definition}
\newtheorem{definition}[theorem]{Definition}

\theoremstyle{definition}
\newtheorem{example}[theorem]{Example}

\theoremstyle{definition}
\newtheorem{remark}[theorem]{Remark}

\theoremstyle{definition}

\theoremstyle{definition}

\theoremstyle{definition}

\newcommand{\Sym}{\text{Sym}}

\newcommand{\op}{\mathrm{op}}
\newcommand{\lf}{\mathrm{lf}}
\newcommand{\fg}{\mathrm{fg}}
\renewcommand{\Vec}{\text{Vec}}
\newcommand{\End}{\text{End}}
\newcommand{\tors}{\mathrm{tors}}
\newcommand{\gen}{\mathrm{gen}}
\newcommand{\pol}{\mathrm{pol}}
\newcommand{\Tr}{\mathrm{Tr}}

\newcommand{\bfV}{\mathbf{V}}

\newcommand{\thrmD}{\mathrm{D}}
\newcommand{\thrmK}{\mathrm{K}}
\newcommand{\thrmR}{\mathrm{R}}

\newcommand{\C}{\mathbf{C}}
\newcommand{\N}{\mathbf{N}}
\newcommand{\Q}{\mathbf{Q}}

\renewcommand{\P}{\mathbf{P}}
\newcommand{\GL}{\mathbf{GL}}

\newcommand{\frakm}{\mathfrak{m}}
\newcommand{\frakg}{\mathfrak{g}}
\newcommand{\frakp}{\mathfrak{p}}

\newcommand{\fraka}{\mathfrak{a}}
\newcommand{\frakb}{\mathfrak{b}}

\newcommand{\frakS}{\mathfrak{S}}
\newcommand{\FI}{\mathbf{FI}}
\newcommand{\FB}{\mathbf{FB}}

\newcommand{\bfS}{\textbf{S}}
\newcommand{\bfA}{\textbf{A}}

\newcommand{\bfM}{\textbf{M}}

\newcommand{\bfL}{\textbf{L}}

\newcommand{\id}{\text{id}}

\newcommand{\Tor}{\text{Tor}}

\newcommand{\scrC}{\mathscr{C}}

\newcommand{\scrP}{\mathscr{P}}

\newcommand{\triv}{\text{triv}}
\newcommand{\Rep}{\text{Rep}}

\newcommand{\Frac}{\text{Frac}}
\newcommand{\Spec}{\text{Spec}\,}

\newcommand{\Hom}{\text{Hom}}
\newcommand{\Mod}{\text{Mod}}

\newcommand{\Ext}{\text{Ext}}

\newcommand{\interior}[1]{%
  {\kern0pt#1}^{\mathrm{o}}%
}

\DeclarePairedDelimiter\abs{\lvert}{\rvert}

\title[Parabolic-equivariant modules]{Parabolic-equivariant modules over polynomial rings in infinitely many variables}
\author{Teresa Yu}
\address{Department of Mathematics, University of Michigan, Ann Arbor, MI}
\email{\href{mailto:twyu@umich.edu}{twyu@umich.edu}}
\urladdr{\url{https://sites.google.com/view/teresayu}}
\thanks{TY was supported by NSF grant DGE-2241144.}

\begin{document}

\begin{abstract}
We study the category of $\P$-equivariant modules over the infinite variable polynomial ring, where $\P$ denotes the subgroup of the infinite general linear group $\GL(\C^\infty)$ consisting of elements fixing a flag in $\C^\infty$ with each graded piece infinite-dimensional.
We decompose the category into simpler pieces that can be described combinatorially, and prove
a number of finiteness results, such as finite generation of local cohomology and rationality of Hilbert series. Furthermore, we show that this category is equivalent to the category of representations of a particular combinatorial category generalizing $\FI$.
\end{abstract}

\maketitle

\section{Introduction}

Equivariant infinite-dimensional commutative algebra has arisen in close connection to the theory of twisted commutative algebras (tca's) and representation stability. Despite being recent developments, the results and ideas from these fields have found far-reaching impacts, such as in algebraic topology \cite{CEF,MW}, commutative algebra \cite{DLL,ESS}, number theory \cite{MNP}, and algebraic statistics \cite{HS,DE}.

There are still relatively few well-understood examples in the field of equivariant infinite-dimensional commutative algebra, the most notable being the $\GL$-algebras $\Sym(\C^\infty\otimes\C^d)$ and $\Sym(\Sym^2(\C^\infty))$ \cite{SSglI,SSglII,NSSnoethdeg2}. Our main contribution is a thorough study of modules over infinite variable polynomial rings that have compatible actions of an infinite parabolic subgroup of the infinite general linear group. These parabolic subgroups have not previously been considered in this context, and the techniques used in this paper highlight how to extend ideas from the study of classical tca's to that of ``multivariate" tca's with additional symmetries. Our primary motivation for studying these equivariant modules comes from the study of $\frakS_\infty$-equivariant modules over the infinite variable polynomial ring; we discuss this further in \S\ref{subsubsec:smod}.

\subsection{Main results}

Let $\bfV=\C^\infty$ be an infinite-dimensional vector space equipped with a flag of length $n\geq 1$ where each graded piece is infinite-dimensional. Let $\P\subset\GL(\bfV)$ be the parabolic subgroup of elements stabilizing this flag. Concretely, $\P$ consists of invertible $(n\times n)$-block upper triangular matrices, so matrices along the diagonal are in $\GL(\C^\infty)$. Let 
\[\bfA=\Sym(\bfV)=\C[x_{ij}:i\in[n],j\geq 1]\] be the infinite variable polynomial ring. An $\bfA$-module $M$ is a module over the ring $\bfA$ equipped with a compatible structure of a polynomial representation of $\P$  (see \S\ref{subsec:prelim Amods} for a precise definition). Our goal is to study the commutative algebra of $\bfA$ by understanding the structure of the category $\Mod_\bfA$ of $\bfA$-modules.

In finite-dimensional commutative algebra, the prime spectrum of a ring provides a good starting point for the study of modules over the ring. The analogous space to study for $\bfA$-modules is the spectrum of $\P$-prime ideals of $\bfA$ (defined in \S\ref{subsec:prelim spec}). We find that the equivariant spectrum is quite simple, as there are exactly $n+1$ such $\P$-prime ideals that are of the form $\frakp_i=(x_{jk}:1\leq j\leq n-i,k\geq 1)$, for $i=0,\ldots,n$.

\begin{theorem}
    The $\P$-prime ideals of $\bfA$ form a single chain of length $n$ in $\bfA$:
    \[0=\frakp_n\subsetneq\frakp_{n-1}\subsetneq\cdots\subsetneq\frakp_0\subsetneq\bfA.\]
\end{theorem}

This result follows from the polynomial representation theory of $\P$, which we studied extensively in \cite{Yu}. In particular, representations of $\P$ are not generally semisimple, and so there are much fewer indecomposable $\P$-submodules of $\bfA$ compared to indecomposable submodules with respect to the Levi subgroup.

The equivariant spectrum provides a way to decompose $\Mod_\bfA$. Let $\Mod_{\bfA,\leq d}$ denote the full subcategory of modules supported on $\frakp_d$, i.e., locally annihilated by a power of $\frakp_d$. Then we have a filtration of $\Mod_\bfA$:
\[\Mod_{\bfA,\leq 0}\subset\Mod_{\bfA,\leq 1}\subset\cdots\subset\Mod_{\bfA,\leq n}=\Mod_{\bfA}.\]
Our goal is to understand each graded piece of the filtration, given by the Serre quotient
\[\Mod_{\bfA,d}=\Mod_{\bfA,\leq d}/\Mod_{\bfA,\leq d-1}.\]
The use of this filtration to understand $\Mod_\bfA$ closely follows the rank stratification used by Sam--Snowden to study $\GL$-equivariant modules over infinite variable polynomial rings \cite{SSglII}. However, we are able to more directly describe a good approximation of the graded pieces of the filtration by studying the Serre quotient $\Mod_d^\gen$ of the category of $\bfA/\frakp_d$-modules by the subcategory of torsion modules.

To describe these quotient categories, we show that they are equivalent to the categories $\Rep(H_d)$ of polynomial representations of subgroups $H_d\subset\P$ that stabilize certain points in the classical spectrum of $\bfA$. We provide a combinatorial description for $\Rep(H_d)$ which allows us to deduce finiteness results for $\Mod_d^\gen$.

\begin{theorem}
    For each $d\in[n]$, the category $\Mod_d^\gen$ is locally of finite length, and finite length objects have finite injective dimension.
\end{theorem}

Although $\Mod^\gen_d$ is smaller than $\Mod_{\bfA,d}$, the categories are sufficiently close for us to understand $\Mod_{\bfA,d}$ from this description. We are then able to study $\Mod_\bfA$ more broadly from these graded pieces.
To accomplish this, we study local cohomology $\thrmR^i\Gamma_d$ and derived saturation $\thrmR^i\Sigma_d$ with respect to each prime $\frakp_d$ (see \S\ref{sec:ModA} for definitions), and show that they preserve finite generation.

\begin{theorem}
    For each $d=0,\ldots,n$ and any finitely generated $\bfA$-module $M$, the local cohomology modules $\thrmR^i\Gamma_{>d}(M)$ and derived saturation modules $\thrmR^i\Sigma_{>d}(M)$ are finitely generated for all $i$, and they vanish for $i\gg 0$.
\end{theorem}

This theorem is instrumental in proving our main structural results on the category $\Mod_\bfA$ and the bounded derived category $\thrmD^b_\fg(\bfA)$ of finitely generated $\bfA$-modules:
\begin{itemize}
    \item We show that the Krull--Gabriel dimension of $\Mod_\bfA$ is $n$, and so any positive integer can be realized as the dimension of a category studied in this paper
    (Proposition~\ref{prop:kg dim}).
    \item We describe a semi-orthogonal decomposition of $\thrmD^b_\fg(\bfA)$, where each piece can be described combinatorially via the categories $\Rep(H_d)$ (Corollary~\ref{cor:semiorthog}).
    \item We find a set of generators for $\thrmD^b_\fg(\bfA)$ as a triangulated category (Corollary~\ref{cor:gens derived cat}).
    \item We show that the Grothendieck group $\thrmK(\bfA)$ of $\Mod_\bfA$ is a free module of rank $(n+1)$ over $\Lambda^{\otimes n}$, where $\Lambda$ denotes the ring of symmetric polynomials (Theorem~\ref{thm:groth grp}).
    \item We provide a natural definition for the Hilbert series of an $\bfA$-module, and prove a corresponding rationality result for the Hilbert series that moreover detects the support of the module (Theorem~\ref{thm:hilbert series}).
\end{itemize}

Finally, we define in \S\ref{sec:fi(n)} a combinatorial category $\FI(n)$ that generalizes the category $\FI$ of finite sets and injections to a ``weighted" setting. Thus, $\FI(n)$-modules generalize $\FI$-modules, which were introduced by Church--Ellenberg--Farb and have found numerous applications \cite{CEF}. We show that the category of $\FI(n)$-modules is equivalent to the category $\Mod_\bfA$ of $\bfA$-modules as tensor categories (Theorem~\ref{thm:main currying}). When $n=1$, this recovers the equivalence between $\FI$-modules and $\GL(\C^\infty)$-equivariant modules over $\Sym(\C^\infty)$.

\subsection{Relation to other work}

\subsubsection{Equivariant modules over infinite-dimensional algebras}

Motivated by the ubiquity of $\FI$-modules and their equivalence with $\GL(\C^\infty)$-equivariant modules over the infinite variable polynomial ring $\Sym(\C^\infty)$, Sam--Snowden studied the category of such modules from an algebraic and homological perspective in \cite{SSglI}. This has sparked the systematic study of equivariant modules over infinite-dimensional algebras from the perspective of commutative and homological algebra; see \cite[Table 1]{Gan} for a summary of the literature. The arguments and ideas of this paper are most closely related to those appearing in \cite{NSSnoethdeg2, SSglII, SSspequiv}.

\subsubsection{The infinite symmetric group}\label{subsubsec:smod}

One of the original examples of an infinite-dimensional algebra that is noetherian up to symmetry comes from the infinite variable polynomial ring $R=\C[\xi_1,\xi_2,\ldots]$ with the action of the infinite symmetric group $\frakS_\infty$ \cite{cohen}. However, the structure of $\frakS_\infty$-equivariant modules over $R$ is still quite mysterious. Only recently were the $\frakS_\infty$-stable ideals of $R$ classified by Nagpal--Snowden \cite{NSsym}, and there are still a number of open problems and conjectures on the asymptotic behavior of invariants related to these ideals \cite{LNNRdim,LNNRreg,MR}.

Let $\mathfrak{h}_n=(\xi_i^n:i\geq 1)$ denote the ideal of $R$ generated by the $n^{\text{th}}$ powers of all the variables. There is a natural way to consider an $\FI(n)$-module as a module over $R/\mathfrak{h}_{n+1}$ with a compatabile $\frakS_\infty$-action, and so we have a functor $\Mod_\bfA\to\Mod_{R/\mathfrak{h}_{n+1}}$. We hope the results in this paper, especially those on the structure of $\Mod^\gen_n$, will see use in studying the category of symmetric $R/\mathfrak{h}_{n+1}$-modules and $R$-modules more generally.

\subsubsection{Representation stability and applications} 

Finiteness properties of equivariant infinite-dimensional algebras are intimately linked to the notion of representation stability as introduced by Church--Farb \cite{CF}. One of the most notable examples of representation stability comes from the cohomology of configuration spaces of distinct points: fixing the cohomological degree while varying the number of points, the cohomology groups form a finitely generated $\FI$-module and therefore exhibit representation stability \cite{CEF}. We hope that the weighted generalizations of $\FI$-modules introduced in this paper will have similar applications, perhaps to configuration spaces of multicolored points.

\subsection{Outline} The paper is organized as follows. In \S\ref{sec:prelim}, we review representation theory of $\P$ and provide some preliminary results on the polynomial ring $\bfA$ and $\P$-equivariant $\bfA$-modules. In \S\ref{sec:rep H_d}, we study the categories $\Rep(H_d)$ of representations of the stabilizer subgroups $H_d\subset\P$, and in \S\ref{sec:gen}, we show that these categories are equivalent to the Serre quotient categories $\Mod_{A_d}^\gen$ of generic $A_d$-modules. We then use these equivalences to prove the main structural results on $\bfA$-modules in \S\ref{sec:ModA}. Finally, we show the equivalence between $\bfA$-modules and $\FI(n)$-modules in \S\ref{sec:fi(n)}.

\subsection{Notation and conventions} 
Fix $n\geq 1$ a positive integer. We work over the field $\C$ of complex numbers. A tensor category is a $\C$-linear and symmetric monoidal category. We collect here the most important notation:
\begin{align*}
    \bfV\textbf{ : }&\text{the infinite-dimensional $\C$-vector space with flag $\{\bfV_i\}_{i=0}^n$}\\
    \P\textbf{ : }&\text{the parabolic subgroup of $\GL(\bfV)$ of elements fixing the flag}\\
    \bfA\textbf{ : }&\text{the $\P$-equivariant infinite variable polynomial ring $\Sym(\bfV)$}\\
    \bfS_{\lambda}(-)\textbf{ : }&\text{the Schur functor associated to the partition $\lambda$}\\
    \bfL,P_a,G_a,H_d\textbf{ : }&\text{subgroups of $\P$}\\
    \mathscr{S}_a\textbf{ : }&\text{the functor $\Rep(\P)\to\Rep(P_a)$ taking $G_a$-invariants}\\
    T_{\underline{\lambda}},S_{\underline{\lambda}}\textbf{ : }&\text{the indecomposable injective and simple $\P$-representations}\\
    \frakp_d\textbf{ : }&\text{the $\P$-prime ideals of $\bfA$}\\
    A_d\textbf{ : }&\text{the quotient rings $\bfA/\frakp_d$}\\
    \FB^{\otimes n},\FB(n),\FI(n),\scrC_d\textbf{ : }&\text{combinatorial categories of $[n]$-weighted finite sets}
\end{align*}

\subsection*{Acknowledgements}
The author is grateful to Andrew Snowden for his guidance, as well as for many valuable discussions, suggestions, and comments. The author also thanks Jordan Ellenberg and Karthik Ganapathy for helpful discussions, as well as Austyn Simpson for comments on a draft.

\section{Preliminaries}\label{sec:prelim}

In this section, we provide some background on representations of the group $\P$, and preliminary results and definitions on $\bfA$-modules.

\subsection{Representation theory of $\P$}\label{subsec:prelim Rep(P)}

Let $\GL=\bigcup_{i=1}^\infty\GL_i(\C)$ denote the infinite general linear group with defining representation $\bfV=\C^\infty$. Equip $\bfV$ with a flag $0=\bfV_0\subset\bfV_1\subset\cdots\subset\bfV_n=\bfV$ of length $n$, where each associated graded piece $\bfV_{(i)}=\bfV_i/\bfV_{i-1}=\C^\infty$ is infinite-dimensional, with basis $\{e_{ij}:j\geq 1\}$. Define $\P\subset\GL$ to be the subgroup preserving this flag. We call $\bfV$ equipped with the choice of flag the \emph{standard representation} of $\P$, and say a representation of $\P$ is \emph{polynomial} if it occurs as a subquotient of a (possibly infinite) direct sum of tensor powers of the standard representation. We let $\Rep(\P)$ denote the category of polynomial representations of $\P$. We review some of the important aspects of polynomial representation theory of the group $\P$, which was studied in \cite{Yu}.

\begin{enumerate}
    \item Let $\FB(n)$ denote the following combinatorial category: objects are $[n]$-weighted finite sets, and morphisms are bijections that do not decrease weights. An \emph{$\FB(n)$-module} is a functor $\FB(n)\to\Vec$, and the tensor product of $\FB(n)$-modules is defined by Day convolution. The categories of finite length $\FB(n)$-modules and finite length polynomial $\P$-representations are equivalent as tensor categories \cite[Theorem 1.1]{Yu}.
    \item Let $\bfL\subset\P$ denote the Levi subgroup, so $\bfL\cong\prod_{i=1}^n\GL$. Polynomial representations of $\bfL$ are semisimple, and every polynomial $\P$-representation can be considered as a polynomial $\bfL$-representation via restriction. Let $\FB$ denote the combinatorial category of finite sets and bijections, and let $\FB^{\otimes n}$ denote its $n$-fold product. One can consider $\FB^{\otimes n}$ as a subcategory of $\FB(n)$ using the $[n]$-weighting, and so $\FB(n)$-modules can also be regarded as $\FB^{\otimes n}$-modules. The categories of polynomial $\bfL$-representations and $\FB^{\otimes n}$-modules are equivalent using classical Schur--Weyl duality, i.e., the equivalence between $\FB$-modules and polynomial representations of $\GL$ \cite[\S 5.4]{SStca}.

    Given a polynomial $\bfL$-representation $V$, there is a decomposition of $V$ with respect to the weight spaces of the torus action. We identify weights with $n$-tuples of nonnegative integer sequences $\alpha=((\alpha_{1j}),\ldots,(\alpha_{nj}))$ such that each sequence $(\alpha_{ij})_{j\geq 1}$ is eventually zero. If two weights $\alpha,\beta$ have disjoint supports, i.e., for any given $i\in[n]$ and $j\geq 1$, $\alpha_{ij}$ and $\beta_{ij}$ are not both nonzero, then we say that the weights are \emph{disjoint}.
    \item The indecomposable injective objects $T_{\underline{\lambda}}$ of $\Rep(\P)$ are indexed by $n$-tuples of partitions $\underline{\lambda}=(\lambda^1,\ldots,\lambda^n)$ \cite[Proposition 4.8]{Yu}. The object $T_{\underline{\lambda}}$ is given by
    \[T_{\underline{\lambda}}=\bigotimes_{i=1}^n\bfS_{\lambda^i}(\bfV/\bfV_{n-i}),\]
    where $\bfS_{\lambda^i}(-)$ denotes the Schur functor corresponding to the partition $\lambda^i$. Furthermore, it is the injective envelope of the simple $\P$-representation
    \[S_{\underline{\lambda}}=\bigotimes_{i=1}^n\bfS_{\lambda^i}(\bfV_{n-i+1}/\bfV_{n-i}).\]
    The projective cover of the simple $S_{\underline{\lambda}}$ is the representation $\bigotimes_{i=1}^n\bfS_{\lambda^i}(\bfV_{n-i+1}).$
    Finite length objects have finite injective and projective resolutions.
    
    Note that the simple $\P$-representations are given by simple $\bfL$-representations, and so the Grothendieck ring $\thrmK(\Rep(\P))$ can be identified with the ring $\Lambda^{\otimes n}$, where $\Lambda$ denotes the ring of symmetric polynomials.
    \item Let $a\in\N$ be a nonnegative integer. Let $G_a$ denote the subgroup of $\P$ consisting of block diagonal matrices such that each nonzero block is of the form
\[\begin{pmatrix} \id_a & 0 \\ 0 & * \end{pmatrix},\]
where $\id_a$ denotes the $a\times a$ identity matrix. Let $P_a$ denote the subgroup of $\P$ consisting of matrices where the $(i,j)$-blocks are each block matrices of the form
\[\begin{pmatrix} * & 0 \\ 0 & \id \end{pmatrix}\quad\text{if }i=j,\quad\begin{pmatrix} * & 0 \\ 0 & 0\end{pmatrix}\quad\text{if }i<j,\]
where the top left blocks are of size $a\times a$ and the bottom right blocks are of size $(\infty-a)\times(\infty-a)$ in both cases. We have that $G_a$ and $P_a$ commute with each other, and so we have a left-exact specialization functor $\mathscr{S}_a:\Rep(\P)\to\Rep(P_a)$ defined by taking $G_a$ invariants, $V\mapsto V^{G_a}$. By \cite[Lemma 3.1]{Yu}, this is a tensor functor. Furthermore, for any element $v\in V$, we have that $v\in\mathscr{S}_a(V)$ for any $a\gg 0$, and the $\bfL$-weight of any element in $\mathscr{S}_a(V)$ is given by an $n$-tuple $\alpha=((\alpha_{ij}))$ of sequences such that for each $i$, the sequence $(\alpha_{ij})_{j\geq 1}$ is $0$ for $j>a$. 
\item Let $\End(\bfV)$ be the monoid of endomorphisms of $\bfV$ that preserve the flag structure on $\bfV$, i.e., linear maps $f:\bfV\to\bfV$ such that $f(\bfV_i)\subset\bfV_i$ for all $i$. For any polynomial representation $V$ of $\P$, there is an induced action of $\End(\bfV)$. In particular, let $a\geq 1$, and let $W\subset V$ be the subspace spanned by elements with $\bfL$-weight disjoint from the weight $((1^a),\ldots,(1^a))$. Then if $g,h\in\P$, there exists an endomorphism $f$ whose action on $\mathscr{S}_a(V)$ agrees with that of $g$ on $\mathscr{S}_a(V)$, while the action of $f$ on elements in $W$ agrees with that of $h$.
\end{enumerate}

\subsection{$\Mod_\bfA$ and noetherianity}\label{subsec:prelim Amods}

Let $\bfA$ denote the commutative algebra object $\Sym(\bfV)$ in $\Rep(\P)$. We consider $\bfA$ as an infinite variable polynomial ring over $\C$ with variables $\{x_{ij}:i\in[n],j\geq 1\}$, where the corresponding basis for the associated graded piece $\bfV_i/\bfV_{i-1}$ is given by the variables $\{x_{ij}:j\geq 1\}$. We let $\abs{\bfA}$ denote the same infinite variable polynomial ring but without the $\P$-representation structure.

An $\bfA$\emph{-module} $M$ is a module object in $\Rep(\P)$, i.e., $M$ is a polynomial $\P$-representation equipped with a multiplication map $\bfA\otimes M\to M$ that is equivariant with respect to $\P$. Let $\Mod_\bfA$ denote the category of all $\bfA$-modules. An $\bfA$-module $M$ is said to be \emph{finitely generated} if there exist finitely many elements such that their $\P$-orbits generate $M$ as an $\abs{\bfA}$-module. 
If $M$ is a finitely generated $\bfA$-module, then there exists a surjection $\bfA\otimes V\to M$, where $V$ is a polynomial $\P$-representation of finite length. We say an $\bfA$-module $M$ is \emph{noetherian} if any submodule is finitely generated.

Another finiteness condition for $\bfA$-modules is defined using the structure of being a polynomial representation of $\bfL$. Polynomial representations of $\bfL$ are semisimple, and since the simple $\P$-representations are also the simple $\bfL$-representations, such representations are also indexed by $n$-tuples $\underline{\lambda}=(\lambda^1,\ldots,\lambda^n)$ of partitions. We say a polynomial $\bfL$-representation $V$ is \emph{bounded} if there exists an integer $s$ such that for any simple representation $S_{\underline{\lambda}}$ appearing as a factor in $V$, we have that $\lambda^i_t=0$ for all $i\in[n]$ and $t> s$, i.e., the length of each partition $\lambda^i$ is at most $s$. If $V$ is bounded, we let $\ell(V)$ denote the minimum such $s$

\begin{example}
\begin{enumerate}
    \item Any finite length $\bfL$-representation is bounded.
    \item     As a polynomial $\bfL$-representation, the algebra $\bfA$ is bounded: every simple object appearing is indexed by a tuple of partitions of the form $((a_1),\ldots,(a_n))$ with $a_i\in\N$, and so $\ell(\bfA)=1$.
    \item     If $V,W$ are bounded representations, then one sees that their tensor product $V\otimes W$ is also bounded using the Littlewood--Richardson rule. Therefore, any finitely generated $\bfA$-module $M$ is bounded, since it is a quotient of a bounded representation $\bfA\otimes V$ with $V$ of finite length.
\end{enumerate}
\end{example}

\begin{lemma}\label{lem:P-equiv implies L-equiv}
    Let $M$ be a finitely generated ($\P$-equivariant) $\bfA$-module. Then it is finitely generated as an $\bfL$-equivariant module of $\bfA$.
\end{lemma}

\begin{proof}
    By \cite[Theorem 1.1]{Yu}, the categories of finite length polynomial $\P$-representations and finite length $\FB(n)$-modules are equivalent. A finite length $\FB(n)$-module is also finite length as a module over the subcategory $\FB^{\otimes n}\subset\FB(n)$. Thus, a finite length polynomial $\P$-representation also has finite length as a representation of $\bfL\subset\P$.

    Now, suppose $M$ is a finitely generated $\bfA$-module. Then $M$ is the quotient of an $\bfA$-module of the form $\bfA\otimes V$, where $V$ is a finite length polynomial $\P$-representation. By the previous paragraph, $V$ is also a finite length polynomial $\bfL$-representation. Thus, $M$ is also finitely generated as an $\bfL$-equivariant module of $\bfA$.
\end{proof}

\begin{proposition}
        $\bfA$ is noetherian as an algebra, i.e., every finitely generated $\bfA$-module is noetherian.
\end{proposition}

\begin{proof}
    The ring $\bfA$ is finitely generated as an algebra object in the category of polynomial $\bfL$-representations, and it is essentially bounded. Then by the multivariate version of \cite[Proposition 2.4]{NSSnoethdeg2}, every finitely generated $\bfL$-equivariant module of $\bfA$ is noetherian. If $M$ is a finitely generated $\P$-equivariant module of $\bfA$, then it is finitely generated as an $\bfL$-equivariant module by Lemma~\ref{lem:P-equiv implies L-equiv}. Furthermore, any chain of $\P$-equivariant submodules of $M$ is also a chain of $\bfL$-equivariant submodules. Thus, $M$ as a $\P$-equivariant module satisfies the ascending chain condition.
\end{proof}

\subsection{Equivariant ideals and $\Spec_\P(\bfA)$}\label{subsec:prelim spec}

A \emph{$\P$-ideal} of $\bfA$ is an $\bfA$-submodule of $\bfA$. Equivalently, it is an ideal of the ring $\abs{\bfA}$ in the usual sense that is stable under the action of $\P$. A $\P$-ideal $\frakp$ is said to be a \emph{$\P$-prime ideal} if for any $\P$-ideals $\fraka,\frakb$ such that $\fraka\frakb\subset\frakp$, we have that either $\fraka\subset\frakp$ or $\frakb\subset\frakp$.

For $i=0,\ldots,n$, let $\frakp_i\subset\bfA$ be the ideal
\[\frakp_i=(x_{jk}:1\leq j\leq n-i,k\geq 1),\]
so $\frakp_n$ is the zero ideal and $\frakp_0$ is the homogeneous maximal ideal of $\bfA$.
In particular, we have that
\[0=\frakp_n\subset\frakp_{n-1}\subset\cdots\subset\frakp_0\subset\bfA.\] 
It is clear that sums of powers of these ideals are $\P$-ideals. The following result shows that all $\P$-ideals are of this form, and that the $\frakp_i$'s constitute all of the $\P$-prime ideals of $\bfA$.

\begin{theorem}\label{thm:A ideals}
    Every $\P$-ideal is of the form $\frakp_{i_1}^{a_1}+\cdots+\frakp_{i_m}^{a_m}$, with $n\geq i_1>\cdots>i_m\geq 0$ and $1\leq a_1<\cdots<a_m$. Every $\P$-prime ideal is of the form $\frakp_i$.
\end{theorem}

\begin{proof}
    Every indecomposable $\P$-submodule of $\bfA=\Sym(\bfV)$ is of the form $\Sym^a(\bfV_i)$, and so all $\P$-ideals are given by sums of powers of the $\frakp_i$'s. If $n\geq i>j\geq 0$ and $a\geq b\geq 1$, then $\frakp_i^a+\frakp_j^b=\frakp_j^b$. Thus, all $\P$-ideals can be written in the desired form.

    Using this characterization of $\P$-ideals, one sees that the $\frakp_i$'s are indeed $\P$-prime. Furthermore, these are all of the $\P$-primes, since $\P$-prime ideals are prime in the usual sense.
\end{proof}

We define the $\P$\emph{-equivariant spectrum} $\Spec_\P(\bfA)$ to be the set of $\P$-prime ideals of $\bfA$ equipped with the Zariski topology. The following corollary will be used to compute the Krull--Gabriel dimension of the category $\Mod_\bfA$ in \S\ref{sec:gen}.

\begin{corollary}\label{cor:dim specA}
    The Krull dimension of $\Spec_\P(\bfA)$ is $n$.
\end{corollary}

\begin{remark}
    For a ring $A$ that may or may not be $\P$-equivariant, we use $\Spec(A)$ to denote the usual spectrum of $A$.
\end{remark}

\subsection{Torsion modules}

For $d=0,\ldots,n$, let $A_d=\bfA/\frakp_d$. We say an $A_d$-module $M$ is \emph{torsion} if every element of $M$ is annihilated by a nonzero element of $A_d$. To understand the structure of $\Mod_\bfA$, we will study the category $\Mod_{A_d}$ of $A_d$-modules and its Serre quotient by the subcategory of torsion modules. The following result gives another characterization of torsion modules.

\begin{proposition}\label{prop:torsion}
    Let $M$ be a finitely generated $A_d$-module. Then $M$ is torsion if and only if it is annihilated by a nonzero $\P$-ideal.
\end{proposition}

\begin{proof}
It is clear that if $M$ is annihilated by a $\P$-ideal, then it is torsion. Now assume $M$ is nonzero and torsion. Given $x\in M$ nonzero, let $a\in A_d$ nonzero be such that $ax=0$. We claim there exists a $k$ such that $a^k(gx)=0$ for all $g\in\P$. If this holds, then the $\P$-orbit of $a^k$ kills $x$, and so the nonzero $\P$-ideal generated by $a^k$ kills $x$. Then, if $M$ is finitely generated by elements $x_1,\ldots,x_r$, one can find a single nonzero element $a\in A_d$ such that $ax_i=0$ for each $i$; here, $a\neq 0$ since $\abs{A_d}$ is a domain. By the claim, there exists a single $k\gg 0$ for which the nonzero $\P$-ideal generated by $a^k$ kills each $x_i$, and so this $\P$-ideal also annihilates $M$.

Let $\frakg$ denote the Lie algebra of $\P$, and let $\mathscr{U}(\frakg)$ be the universal enveloping algebra. We first show by induction on $k$ that if $X_1,\ldots,X_k\in\frakg$, then $a^{k+1}X_k\cdots X_1x=0$. The $k=0$ case holds by assumption that $ax=0$, so suppose $a^kX_{k-1}\cdots X_1x=0$, and apply $X_k$ to this product:
\begin{align*}
    0&=(X_k(a^k))(X_{k-1}\cdots X_1x)+a^k(X_kX_{k-1}\cdots X_1x)\\
    &=ka^{k-1}(X_ka)(X_{k-1}\cdots X_1x)+a^kX_k\cdots X_1x.
\end{align*}
Then, multiplying by $a$, we see that the first term is $0$ by induction (note that $a^k,X_ka\in A_d$ and therefore commute), and so $a^{k+1}X_k\cdots X_1x=0$ as desired.

Let $V\subset M$ be the $\P$-representation generated by $x$, and let $b\geq 1$ be such that $a\in\mathscr{S}_b(A_d)$. Using the action of the Levi subgroup $\bfL\subset\P$, one sees that there exists $y\in V$ such that the weight of $y$ with respect to $\bfL$ is disjoint from that of $a$, and that $y$ also generates $V$ as a $\P$-representation. 
Let $X\in\mathscr{U}(\frakg)$ be such that $Xx=y$, and so by the above, there exists $k$ such that $a^ky=0$. We now show that this same $k$ is such that $a^k(gx)=0$ for all $g\in\P$.

Fix $g\in\P$. Since $a,y$ have disjoint $\bfL$-weights, there exists an endomorphism $F_g:\bfV\to\bfV$ preserving the flag structure on $\bfV$, and such that the action of $F_g$ on $y$ agrees with that of $g$, 
and the action of $F_g$ on $\mathscr{S}_b(\bfV)$ agrees with that of the identity. In particular, $F_g(a)=a$, while $F_g(y)=gy$. Then,
\[0=F_g(a^ky)=F_g(a^k)F_g(y)=a^k(gy).\]
Every element in the $\P$-orbit of $x$ can be realized as an element in the $\P$-orbit of $y$, and so the proof is complete.
\end{proof}

\section{Representations of stabilizer subgroups}\label{sec:rep H_d}

For each $d\in[n]$, let $\xi_d:\bfV\to\C$ be the linear form defined by
\[\xi_d(e_{ij})=\begin{cases}
1\quad&(i,j)=(n-d+1,1),\\
0\quad&\text{otherwise}.
\end{cases}\]
Let $H_d\subset\P$ be the subgroup that stabilizes $\xi_d$. Concretely, $H_d$ consists of matrices of $\P$ such that the $(n-d+1,n-d+1)$-block is an element of the infinite general affine group $\textbf{GA}(\bfV_{(n-d+1)})$ whose top row is zero except for a one in the upper left corner, and whose top rows of the $(n-d+1,i)$-blocks for $i> n-d+1$ are all zero.

A representation $V$ of $H_d$ is \emph{polynomial} if it occurs as a subquotient of a (possibly infinite) direct sum of tensor powers of $\bfV$. In this section, we describe the category $\Rep(H_d)$ of polynomial $H_d$-representations.

\subsection{Combinatorial description}

For each $d\in[n]$, we define a combinatorial category $\scrC_d$ as follows. Objects are $[n]$-weighted finite sets $S=\bigsqcup_{i=1}^n S_i$, where the elements of the subset $S_i$ have weight $i$. A morphism is an injection of sets $\varphi:S\to T$ such that weights do not decrease, i.e., $\varphi(S_i)\subset\bigsqcup_{j=i}^n T_j$, and $\varphi$ is surjective on $\bigsqcup_{i=1}^{d-1} T_i$. A \emph{$\scrC_d$-module} is a functor $\scrC_d\to\Vec$. Let $\Mod_{\scrC_d}$ denote the category of $\scrC_d$-modules, and let $\Mod_{\scrC_d}^{\mathrm{lf}}$ denote the full subcategory spanned by locally finite length objects. We will show that $\Mod^\lf_{\scrC_d}$ is equivalent to $\Rep(H_d)$. Before showing this equivalence, we describe some important aspects of this category.

Let $\underline{a}\in\N^n$ be an $n$-tuple of nonnegative integers. There is a weighted set with elements of weight $i$ given by $[a_i]=\{1,\ldots,a_i\}$; we also use $\underline{a}$ to denote this object. Every object of $\scrC_d$ is isomorphic to a unique $\underline{a}$, and the group of automorphisms of $\underline{a}$ is the group $\frakS_{a_1}\times\cdots\times\frakS_{a_n}$, which we denote by $\frakS_{\underline{a}}$. 

For a tuple of partitions $\underline{\lambda}$ with $\underline{a}=(\abs{\lambda^i})$, we have the simple $\scrC_d$-module $\bfM_{\underline{\lambda}}$ given by
\[\underline{b}\mapsto\begin{cases}
    M_{\lambda^1}\boxtimes\cdots\boxtimes M_{\lambda^n}&\quad \text{if $\underline{b}=\underline{a}$,}\\
    0&\quad\text{otherwise},
\end{cases}\]
where $M_{\lambda^i}$ denotes the simple $\frakS_{a_i}$-module corresponding to the partition $\lambda^i$. All simple $\scrC_d$-modules are of this form. We also use $\bfM_{\underline{\lambda}}$ to denote the corresponding simple $\frakS_{\underline{a}}$-module.

The category $\scrC_d$ is \emph{inwards finite}, meaning for every object $\underline{a}$, there are only finitely many objects $\underline{b}$ up to isomorphism such that there exists a morphism $\underline{b}\to\underline{a}$ in $\scrC_d$. This implies that the injective envelope of a simple object has finite length, and every $\scrC_d$-module of finite length has finite injective dimension \cite[\S 2.1.5]{SSstabpatterns}.

\subsection{Simple representations}

Let $W_d\subset\bfV$ denote the $H_d$-submodule $W_d=\ker\xi_d\cap\bfV_{n-d+1}$. For an $n$-tuple $\underline{\lambda}$, consider the following polynomial $H_d$-representation:
\begin{align*}
S_{d,\underline{\lambda}}&=\left(\bigotimes_{i\neq d}\bfS_{\lambda^i}(\bfV_{(n-i+1)})\right)\otimes \bfS_{\lambda^{d}}(W_d/\bfV_{n-d}).
\end{align*}
It is irreducible, as it is irreducible as a representation of the Levi subgroup $L_d\subset H_d$.
The goal of this subsection is to show that these are all of the simple $H_d$-modules.

Consider the opposite category $\scrC^\op_d$ of the combinatorial category described above. We define a functor $\mathscr{T}_d:\scrC_d^\op\to\Rep(H_d)$ as follows. For an object $\underline{a}$, define the corresponding polynomial $H_d$-representation, denoted $T_{d,\underline{a}}$, by
\[T_{d,\underline{a}}=\bigotimes_{i=1}^n(\bfV/\bfV_{n-i})^{\otimes a_i}.\]
Given a morphism $\varphi:\underline{a}\to\underline{b}$ in $\scrC_d^\op$, we define $T_{d,\underline{a}}\to T_{d,\underline{b}}$ as follows. If $x\in\underline{a}$ is an element of weight $i$ for which $\varphi(x)$ is defined and is of weight $j\leq i$, then map the tensor factor $\bfV/\bfV_{n-i}$ of $T_{d,\underline{a}}$ corresponding to $x$ to the tensor factor $\bfV/\bfV_{n-j}$ of $T_{d,\underline{b}}$ corresponding to $\varphi(x)$ via the natural quotient. If $\varphi(x)$ is not defined, then $x$ is of weight $i\geq d$ and so $n-i\leq n-d$. The linear functional $\xi_d$ is therefore well-defined on $\bfV/\bfV_{n-i}$ and we apply it to the corresponding tensor factor.

\begin{example}
    Let $n=3$ and $d=3$. Consider the objects $\underline{b}=(1,2,0)$ and $\underline{a}=(0,2,2)$ in $\scrC_3$, and define a morphism $\underline{b}\to\underline{a}$ by mapping an element of weight $2$ in $\underline{b}$ to an element of weight $3$ in $\underline{a}$, and mapping the other two elements of $\underline{b}$ to the elements of weight $2$ in $\underline{a}$. Let $\varphi:\underline{a}\to\underline{b}$ be the induced morphism in $\scrC_3^\op$. Then the corresponding map of polynomial $H_3$-representations is given by
    \[(\bfV/\bfV_1)\otimes(\bfV/\bfV_1)\otimes\bfV\otimes\bfV\to(\bfV/\bfV_2)\otimes(\bfV/\bfV_1)\otimes(\bfV/\bfV_1),\]
    where the first three tensor factors of the left side map onto the right side via the maps $\bfV/\bfV_1\to\bfV/\bfV_2$, $\bfV/\bfV_1\to\bfV/\bfV_1$, and $\bfV\to\bfV/\bfV_1$, while the linear functional $\xi_3:\bfV\to\C$ is applied to the last tensor factor $\bfV$.
\end{example}

For a tuple $\underline{a}$, define the polynomial $H_d$-representation $K_{d,\underline{a}}$ by
\[K_{d,\underline{a}}=\bigcap\ker(T_{d,\underline{a}}\to T_{d,\underline{b}}),\]
where the intersection is over all non-isomorphisms $\underline{a}\to\underline{b}$ in $\scrC_d^\op$; note that this intersection is finite since $\scrC_d$ is inwards finite.

\begin{lemma}\label{lem:H_d kernel}
Let $\underline{a}$ be a tuple. Then
        \[K_{d,\underline{a}}=\left(\bigotimes_{i\neq d}\bfV_{(n-i+1)}^{\otimes a_i}\right)\otimes (W_d/\bfV_{n-d})^{\otimes a_{d}}.\]
\end{lemma}

\begin{proof}
    First, let $j\in[n-1]$ and $\underline{b}=(a_1,\ldots,a_{j-1},0,a_j+a_{j+1},a_{j+2},\ldots,a_n)$. Consider the following morphism $\varphi_j:\underline{a}\to\underline{b}$ in $\scrC^\op_d$: it is a bijection as a map of sets, and it is the identity on elements of weight not equal to $j$. Then the kernel of the surjection $\varphi_j:T_{d,\underline{a}}\to T_{d,\underline{b}}$ is
    \[\left(\bigotimes_{i\neq j}(\bfV/\bfV_{n-i})^{\otimes a_i}\right)\otimes \bfV_{(n-j+1)}^{\otimes a_j}.\]
    Now let $\underline{c}=(a_1,\ldots,a_{d-1},0,a_{d+1},\ldots,a_n)$, and consider the morphism $\psi:\underline{a}\to\underline{c}$ in $\scrC^\op_d$ given by the identity on elements of weight not equal to $d$. Then the surjection $\psi:T_{d,\underline{a}}\to T_{d,\underline{c}}$ has kernel
    \[\left(\bigotimes_{i\neq d}(\bfV/\bfV_{n-i})^{\otimes a_i}\right)\otimes (W_d/\bfV_{n-d})^{\otimes a_{d}}.\]
    Thus,
    \[K_{d,\underline{a}}\subset \ker\psi\cap\bigcap_j\ker\varphi_j=\left(\bigotimes_{i\neq d}(\bfV_{n-i+1}/\bfV_{n-i})^{\otimes a_i}\right)\otimes (W_d/\bfV_{n-d})^{\otimes a_{d}}.\]

    The vector
    \[v=\left(\bigotimes_{i\neq d}e_{n-i+1,1}\otimes\cdots\otimes e_{n-i+1,a_i}\right)\otimes e_{n-d+1,2}\otimes\cdots\otimes e_{n-d+1,a_{d}+1}\]
    generates the right-hand side as an $H_d$-module, and this vector has weight
    \[(1^{a_n},\ldots,1^{a_{d+1}},0,1^{a_{d}},\ldots,1^{a_1})\]
    under the action of the Levi subgroup $L_d$. Thus, if $T_{d,\underline{b}}$ does not have a weight space of this weight, then the right-hand side is in the kernel of any map $T_{d,\underline{a}}\to T_{d,\underline{b}}$.
    
    We now show that for any $\underline{b}$ with $\varphi:\underline{a}\to\underline{b}$ a non-isomorphism in $\scrC^\op_d$, there is indeed no weight space of this weight in $T_{d,\underline{b}}$. If $\abs{\underline{b}}<\abs{\underline{a}}$, then this is clearly true. Otherwise, we must have that $\abs{\underline{a}}=\abs{\underline{b}}$ with $(a_n,\ldots,a_1)>(b_n,\ldots,b_1)$ under the dominance order, and so there is no element of this weight in this case as well.
\end{proof}

\begin{remark}\label{rmk:H_d kernel min maps}
    Every morphism in $\scrC^\op_d$ can be written as a composition of morphisms of the following two forms:
    \begin{enumerate}
        \item a bijection $\underline{a}\to\underline{b}$ where $\abs{\underline{a}}=\abs{\underline{b}}$ and $(a_n,\ldots,a_1)>(b_n,\ldots,b_1)$ is a cover relation;
        \item a morphism $\underline{a}\to\underline{b}$, where
        \[b_i=\begin{cases}
        a_i&\quad i\neq d,\\
        a_{d}-1&\quad i=d,
        \end{cases}\]
        and the morphism corresponds to the natural injection $\underline{b}\to\underline{a}$ in $\scrC_d$.
    \end{enumerate}
    Thus, $K_{d,\underline{a}}$ can also be given as the intersection of kernels of maps $T_{d,\underline{a}}\to T_{d,\underline{b}}$, where $\underline{a}\to\underline{b}$ is of one of the above two forms.
\end{remark}

\begin{proposition}
    All irreducible constituents of $T_{d,\underline{a}}$ are of the form $S_{d,\underline{\lambda}}$.
    In particular, the $S_{d,\underline{\lambda}}$'s constitute all irreducible objects in $\Rep(H_d)$.
\end{proposition}

\begin{proof}
We have an exact sequence
    \[0\to K_{d,\underline{a}}\to T_{d,\underline{a}}\to\bigoplus_{\underline{a}\to\underline{b}}T_{d,\underline{b}},\]
    where the direct sum ranges over all morphisms $\underline{a}\to\underline{b}$ described in Remark~\ref{rmk:H_d kernel min maps}. The simple constituents of $K_{d,\underline{a}}$ consist of the $S_{d,\underline{\lambda}}$'s with $\underline{a}=(\abs{\lambda^i})$ by Lemma~\ref{lem:H_d kernel}. We induct on the objects $\underline{b}$, with base case given by $\underline{b}=(b,0,\ldots,0)$ for $b=a_1+\cdots+a_{d-1}$. Then $T_{d,\underline{b}}=\bfV_{(n)}^{\otimes b}$, which has simple constituents of the form $\bfS_{\lambda}(\bfV_{(n)})$. Therefore, all simple constituents of $T_{d,\underline{a}}$ have the desired form.
\end{proof}

\begin{lemma}\label{lem:socle H_d}
    The socle of $T_{d,\underline{a}}$ is $K_{d,\underline{a}}$.
\end{lemma}

\begin{proof}
    Let $U_d\subset H_d$ denote the unipotent radical. Since every simple $H_d$-representation is of the form $S_{d,\underline{\lambda}}$, the action of $U_d$ on any semisimple $H_d$-representation is trivial. 

    We have that $K_{d,\underline{a}}$ is semisimple by Lemma~\ref{lem:H_d kernel}. If $V\subset T_{d,\underline{a}}$ is a submodule that properly contains $K_{d,\underline{a}}$, then $V$ must have a nontrivial action of $U_d$, so $V$ cannot be semisimple.
\end{proof}

\subsection{Indecomposable injectives}

We now show that $\Rep(H_d)$ and $\Mod_{\scrC_d}^{\lf}$ are equivalent, and use this to identify the indecomposable injectives of $\Rep(H_d)$. The functor giving the equivalence is described by the structured tensor product $\odot=\odot_{\scrC_d}$, which is defined in \cite[\S2.1.9]{SSstabpatterns}.

\begin{proposition}\label{prop:repH_d equiv Cd-mod}
    We have an equivalence of categories $\Rep(H_d)\cong\Mod^\lf_{\scrC_d}$.
\end{proposition}

\begin{proof}
We have a functor $\Mod_{\scrC_d}^\lf\to\Rep(H_d)$ given by the structured tensor product $M\mapsto M\odot\mathscr{T}_d$. This functor is cocontinuous, and both categories are locally noetherian and artinian. It therefore suffices to show an equivalence among the subcategories of finite length objects. By \cite[Corollary 2.1.12]{SSstabpatterns}, this functor is an equivalence if the following conditions hold:
\begin{enumerate}
    \item For any simple $\frakS_{\underline{a}}$-module $\bfM_{\underline{\lambda}}$ with $\underline{a}=(\abs{\lambda^i})$, the polynomial $H_d$-representation $\Hom_{\frakS_{\underline{a}}}(\bfM_{\underline{\lambda}},K_{d,\underline{a}})$ is simple.
    \item For every simple $H_d$-representation $S_{d,\underline{\lambda}}$, there is a unique tuple $\underline{a}$ such that the space $\Hom_{H_d}(S_{d,\underline{\lambda}},T_{d,\underline{a}})$ is nonzero, and it is an irreducible representation of $\frakS_{\underline{a}}$.
\end{enumerate}
By Lemma~\ref{lem:H_d kernel} and Schur--Weyl duality, $\Hom_{\frakS_{\underline{a}}}(\bfM_{\underline{\lambda}},K_{d,\underline{a}})=S_{d,\underline{\lambda}}$, 
so (1) holds. Condition (2) holds by Lemma~\ref{lem:socle H_d}, as $S_{d,\underline{\lambda}}$ is a submodule of $T_{d,\underline{a}}$ if and only if $\underline{a}=(\abs{\lambda^i})$, and the Hom space in this case is the simple $\frakS_{\underline{a}}$-representation $\bfM_{\underline{\lambda}}$.
\end{proof}

For a tuple of partitions $\underline{\lambda}$, define the $H_d$-representation $T_{d,\underline{\lambda}}$ to be
\[T_{d,\underline{\lambda}}=\bigotimes_{i=1}^n\bfS_{\lambda^i}(\bfV/\bfV_{n-i}).\]

\begin{corollary}\label{cor:inj H_d}
The injective envelope of $S_{d,\underline{\lambda}}$ is $T_{d,\underline{\lambda}}$, and the $T_{d,\underline{\lambda}}$'s constitute all indecomposable injective objects of $\Rep(H_d)$.
\end{corollary}

\begin{proof}
Let $B\frakS_{\underline{a}}$ denote the category with a single object that has automorphism group $\frakS_{\underline{a}}$. There is a natural fully faithful functor $i:B\frakS_{\underline{a}}\to\scrC^\op_d$, which induces the pullback functor $i^*:\Mod_{\scrC^\op_d}\to\Rep(\frakS_{\underline{a}})$ and its left Kan extension $i_\#:\Rep(\frakS_{\underline{a}})\to\Mod_{\scrC^\op_d}$, which is left adjoint to $i^*$. We have that $\bfM_{\underline{\lambda}}$ is an indecomposable projective object in $\Rep(\frakS_{\underline{a}})$. Then by \cite[\S2.1.5]{SSstabpatterns}, the object $i_\#(\bfM_{\underline{\lambda}})$ is the projective cover of the simple $\scrC^\op_d$-module $\bfM_{\underline{\lambda}}$, and every indecomposable projective object in $\Mod_{\scrC_d^\op}^\lf$ is of this form.

Recall the $\scrC^\op_d$-module $\mathscr{T}_d$ defined by $\underline{a}\mapsto T_{d,\underline{a}}$. This defines a contravariant functor $\Phi:\Mod_{\scrC^\op_d}^\lf\to\Rep(H_d)$ given by $M\mapsto \Hom_{\scrC^\op_d}(M,\mathscr{T}_d)$, and by \cite[Theorem 2.1.11]{SSstabpatterns} and Proposition~\ref{prop:repH_d equiv Cd-mod}, this is an equivalence of categories. Therefore, $\Phi(i_\#(\bfM_{\underline{\lambda}}))$ is the injective envelope of the simple $H_d$-representation $\Phi(\bfM_{\underline{\lambda}})=S_{d,\underline{\lambda}}$. Using the adjunction, we see that
    \begin{align*}
\Phi(i_\#(\bfM_{\underline{\lambda}}))&=\Hom_{\scrC^\op_d}(i_\#(\bfM_{\underline{\lambda}}),\mathscr{T}_d)=\Hom_{\frakS_{\underline{a}}}(\bfM_{\underline{\lambda}},i^*(\mathscr{T}_d))=\Hom_{\frakS_{\underline{a}}}(\bfM_{\underline{\lambda}},T_{d,\underline{a}})=T_{d,\underline{\lambda}}.
    \end{align*}
The result then follows.
\end{proof}

\begin{remark}
    In \cite{Yu}, we showed that all indecomposable injectives in the category $\Rep(\P)$ of finite length polynomial $\P$-representations are also of the form $\bigotimes_{i=1}^n\bfS_{\lambda^i}(\bfV/\bfV_{n-i})$, for a tuple $\underline{\lambda}$ of partitions. Thus, all injective $H_d$-representations are restrictions of injective $\P$-representations. 
\end{remark}

The following key finiteness result follows from the properties of $\Mod_{\scrC_d}$.

\begin{corollary}\label{cor:repH_d finite/inj}
    Let $V$ be a finitely generated object of $\Rep(H_d)$. Then $V$ has finite length and finite injective dimension.
\end{corollary}

\section{Generic categories}\label{sec:gen}

\subsection{Statement of the main theorem} 
For $d=1,\ldots,n$, let $A_d=\bfA/\frakp_{d}$; note that $A_n=\bfA$. Let $\Mod_d^\tors$ be the category of torsion $A_d$-modules, and let $\Mod_d^\gen$ denote the Serre quotient of $\Mod_{A_d}$ by $\Mod_d^\tors$. We call $\Mod_d^\gen$ the \emph{generic category} of $A_d$-modules. Our goal for this section is to understand the structure of $\Mod_d^\gen$ via the category $\Rep(H_d)$ studied in the previous section. In particular, we show that $\Mod_d^\gen\cong\Rep(H_d)$ as tensor categories.

The intuition for the equivalence is as follows. Consider the scheme $\Spec(A_d)$. Then, $A_d$-modules correspond to $\P$-equivariant quasi-coherent sheaves on $\Spec(A_d)$, torsion modules correspond to such sheaves that restrict to zero on a dense open subset $U$, and generic $A_d$-modules correspond to $\P$-equivariant quasi-coherent sheaves on $U$. Recall the linear functional $\xi_d$ defined in \S\ref{sec:rep H_d}; its $\P$-orbit is a dense open subset of $\Spec(A_d)$, and its stabilizer is $H_d\subset\P$. Thus, $\P$-equivariant quasi-coherent sheaves on $U$ should correspond to representations of $H_d$. 

This intuition comes from standard theory on representations of algebraic groups. However, since we are working with infinite-dimensional spaces, the details in our setting are a bit more technical. To prove the equivalence of categories, we define the functor $\Phi_d:\Mod_{A_d}\to\Rep(H_d)$ as follows. 
The form $\xi_d$ induces a ring homomorphism $\abs{A_d}\to\C$, which has kernel $\frakm_d\subset \abs{A_d}$, a maximal (not $\P$-equivariant) ideal. For an $A_d$-module $M$, define $\Phi_d(M)=M/\frakm_d M$. This is indeed a polynomial $H_d$-representation: $\frakm_d$ is $H_d$-equivariant, and since $M$ is a quotient of $A_d\otimes V$ for some polynomial $\P$-representation $V$, then $\Phi(M)$ is a quotient of the polynomial $H_d$-representation $\Phi_d(A_d\otimes V)=V$. Furthermore, $\Phi_d$ is a tensor functor.

The main result of this section is the following.

\begin{theorem}\label{thm:gen mods equiv Hd reps}
    The functor $\Phi_d$ induces an equivalence of tensor categories $\Mod_d^\gen\cong\Rep(H_d)$.
\end{theorem}

Let $\overline{T}_d:\Mod_{A_d}\to\Mod_d^\gen$ denote the localization functor. The following result is an immediate consequence of Theorem~\ref{thm:gen mods equiv Hd reps} and the results in \S\ref{sec:rep H_d}.

\begin{corollary}\label{cor:gen Amod}
     For any $d=1,\ldots,n$, we have the following.
    \begin{enumerate}
        \item If $M$ is a finitely generated $A_d$-module, then $\overline{T}_d(M)$ has finite length in $\Mod_d^\gen$. In particular, $\Mod_d^\gen$ is locally of finite length.
        \item Objects of finite length in $\Mod_d^\gen$ have finite injective dimension.
        \item The injective objects of $\Mod_d^\gen$ are the objects of the form $\overline{T}_d(A_d\otimes V)$, for $V$ an injective polynomial $\P$-representation.
        \item The simple objects of $\Mod_d^\gen$ are the objects of the form $\overline{T}_d(A_d\otimes S_{d,\underline{\lambda}})$, where $S_{d,\underline{\lambda}}$ is a simple $H_d$-representation.
    \end{enumerate}
\end{corollary}

\begin{remark}
Consider $A_0=\bfA/\frakp_0$, which is isomorphic to $\C$ as a vector space. Then $\Mod_{A_0}\cong\Rep(\P)$ are equivalent via the functor $\Rep(\P)\to\Mod_{A_0}:V\mapsto A_0\otimes V$. Thus, $\Mod_{A_0}$ is already locally of finite length, and objects of finite length also have finite injective dimension (recall the background on $\Rep(\P)$ in \S\ref{subsec:prelim Rep(P)}).
\end{remark}

\subsection{Proof of the theorem}
We now prove Theorem~\ref{thm:gen mods equiv Hd reps} by showing that $\Phi_d$ induces an exact tensor functor $\overline{\Phi}_d:\Mod^\gen_d\to\Rep(H_d)$ that is fully faithful and essentially surjective.

For $a\geq 1$, recall from \S\ref{subsec:prelim Rep(P)} the subgroups $G_a,P_a\subset\P$, and the specialization functor $\mathscr{S}_a:\Rep(\P)\to\Rep(P_a)$ taking $G_a$-invariants.

\begin{lemma}\label{lem:Phi_d exact}
    $\Phi_d$ is exact.
\end{lemma}

\begin{proof}
    It suffices to prove that $\Phi_d$ is exact on finitely generated modules, since $\Phi_d$ is cocontinuous. Suppose $a\gg 0$. Then by generic flatness and since the $P_a$-orbit of $\frakm_d\cap \mathscr{S}_a(A_d)$ is an open subset of $\Spec(\mathscr{S}_a(A_d))$, finitely generated $P_a$-equivariant modules over $\mathscr{S}_a(A_d)$ are flat over any point in that open subset, including at $\frakm_d\cap \mathscr{S}_a(A_d)$.
    
    We have that $A_d/\frakm_d$ is the colimit $\varinjlim \mathscr{S}_a(A_d) /(\frakm_d\cap \mathscr{S}_a(A_d))$ over $a$. The lemma then follows from the fact that flatness is preserved under colimits.
\end{proof}

\begin{lemma}\label{lem:Phi_d ker}
    The kernel of $\Phi_d$ is $\Mod_d^\tors$.
\end{lemma}

\begin{proof}
    Suppose $M$ is a finitely generated torsion $A_d$-module. Then by Theorem~\ref{thm:A ideals} and Proposition~\ref{prop:torsion}, there exists some $k>0$ such that $x_{n-d+1,1}^k$ annihilates $M$. Since $x_{n-d+1,1}^k-1\in\frakm_d$, we see that $\frakm_d M=M$, and so $\Phi_d(M)=0$. The argument for when $M$ is not finitely generated also follows.

    Conversely, suppose $M$ is a nonzero $A_d$-module such that $\Phi_d(M)=0$. Since $\Phi_d$ is cocontinuous, we may assume that $M$ is finitely generated. For $a\gg 0$, we see that the point $\frakm_d\cap \mathscr{S}_a(A_d)\in\Spec(\mathscr{S}_a(A_d))$ is not in the support of $\mathscr{S}_a(M)$. Furthermore, since $\mathscr{S}_a(M)$ is $P_a$-equivariant, the $P_a$-orbit of this point is also not in the support. In particular, $\mathscr{S}_a(M)$ is only supported at the homogeneous maximal ideal $\mathscr{S}_a(\frakp_0)\in\Spec(\mathscr{S}_a(A_d))$, and so $M$ is also only supported at $\frakp_0$. This implies that $M$ is torsion.
\end{proof}

Thus, $\Phi_d$ induces an exact tensor functor $\overline{\Phi}_d:\Mod^\gen_d\to\Rep(H_d)$. We show that $\overline{\Phi}_d$ is fully faithful and essentially surjective to prove Theorem~\ref{thm:gen mods equiv Hd reps}.

\begin{lemma}\label{lem:Phi_d full}
    $\overline{\Phi}_d$ is full.
\end{lemma}

\begin{proof}
Suppose $M,N$ are $A_d$-modules, and $f:\Phi_d(M)\to\Phi_d(N)$ is a morphism of polynomial representations of $H_d$. We may further assume without loss of generality that $M,N$ are finitely generated. For $a\gg 0$, let $U\subset\Spec(\mathscr{S}_a(A_d))$ denote the open $P_a$-orbit of $\frakm_d\cap\mathscr{S}_a(A_d)$, and let $\iota:U\hookrightarrow\Spec(\mathscr{S}_a(A_d))$ denote the open immersion. 

Let $M_a=\mathscr{S}_a(M)$ and $N_a=\mathscr{S}_a(N)$. By taking $(G_a\cap H_d)$-invariants, we have a map of $(P_a\cap H_d$)-representations corresponding to a morphism on pullbacks $\iota^*(M_a)\to\iota^*(N_a)$. This induces a map of $\mathscr{S}_a(A_d)$-modules 
\[f_a:M_a\to\iota_*(\iota^*(M_a))\to\iota_*(\iota^*(N_a)).\]
Note that this map is $P_a$-equivariant by functoriality, but $\iota_*(\iota^*(N_a))$ may not be necessarily a polynomial $P_a$-representation. Let $N'_a\subset\iota_*(\iota^*(N_a))$ be the maximal polynomial subrepresentation with respect to the action of the Levi subgroup of $P_a$; in particular, $N'_a$ is a $\mathscr{S}_a(A_d)$-module that contains $N_a$, and the image of the map $M_a$ under the map $M_a\to\iota_*(\iota^*(N_a))$ is contained in $N'_a$, since $M_a$ is a polynomial $P_a$-representation.

Let $N'$ be the canonical $A_d$-module such that $\ell(N')\leq a$ and such that $\mathscr{S}_a(N')=N'_a$. Then $f_a$ determines a map of $A_d$-modules $g:M\to N'$. We now claim that $N'\cong N$ in $\Mod_d^\gen$. Indeed, pulling back $\iota_*(\iota^*(N_a))/N_a$ along $\iota^*$ results in $N_a/N_a=0$ and so
\[\mathscr{S}_a(N')/N_a\subset\iota_*(\iota^*(N_a))/N_a\]
is a torsion module. Since $a\gg 0$, this implies that $N'/N$ is torsion as an $A_d$-module, and so $N'\cong N$ in $\Mod^\gen_d$. Thus, $\Phi_d(g)=f$, and $\overline{\Phi}_d$ is indeed full.
\end{proof}

\begin{lemma}\label{lem:Phi_d faithful}
    $\overline{\Phi}_d$ is faithful.
\end{lemma}

\begin{proof}
    Suppose $f:M\to N$ is a map of $A_d$-modules such that the map of $H_d$-representations $\Phi_d(f):\Phi_d(M)\to\Phi_d(N)$ is zero. Then the image of $f$ is torsion by Lemma~\ref{lem:Phi_d ker}, and so $f=0$ as a map in $\Mod^\gen_d$.
\end{proof}

\begin{lemma}\label{lem:Phi_d ess surj}
    $\overline{\Phi}_d$ is essentially surjective.
\end{lemma}

\begin{proof}
First note that all $H_d$-representations of the form $T_{d,\underline{a}}$ are in the essential image of $\Phi_d$, as
\[\Phi_d(A_d\otimes T_{d,\underline{a}})=A_d/\frakm_d\otimes T_{d,\underline{a}}\cong T_{d,\underline{a}}.\]
By Corollary~\ref{cor:inj H_d}, every finitely generated polynomial $H_d$-representation $M$ can be realized as the kernel of a map $f:S\to T$ between direct sums of representations of the form $T_{d,\underline{a}}$. By Lemma~\ref{lem:Phi_d full}, $f=\overline{\Phi}_d(g)$ for some morphism $g:S'\to T'$ in $\Mod_d^\gen$. By Lemma~\ref{lem:Phi_d exact}, $\overline{\Phi}_d$ is exact, and so $M=\overline{\Phi}_d(\ker g)$. Thus, $M$ is also in the essential image of $\overline{\Phi}_d$.
\end{proof}

\subsection{Krull--Gabriel dimension of $\Mod_\bfA$}

We can use Theorem~\ref{thm:gen mods equiv Hd reps} to calculate the Krull--Gabriel dimension of $\Mod_\bfA$, which is an important invariant of abelian categories. For a locally noetherian abelian category $\mathscr{A}$, we define it as follows. Let $\mathscr{A}^\fg$ be the category of finitely generated objects, and let $\mathscr{A}_0\subset\mathscr{A}^\fg$ be the Serre subcategory on finite length objects. For $i>0$, let $\mathscr{A}_i\subset\mathscr{A}^\fg$ be the subcategory on objects whose images in the Serre quotient category $\mathscr{A}^\fg/\mathscr{A}_{i-1}$ are of finite length. Then the \emph{Krull--Gabriel dimension} of $\mathscr{A}$ is defined to be the minimal $k$ for which $\mathscr{A}_k=\mathscr{A}^\fg$. In particular, if all finitely generated objects of $\mathscr{A}$ are of finite length, then its Krull--Gabriel dimension is $0$.

\begin{proposition}\label{prop:kg dim}
    The Krull--Gabriel dimension of $\Mod_\bfA$ is $n$.
\end{proposition}

\begin{proof}
    By Corollary~\ref{cor:gen Amod}, the Krull--Gabriel dimension of $\Mod^\gen_{d}$ is $0$ for $d=1,\ldots,n$. Every quotient $A/\frakp_d$ for $\frakp_d$ a non-maximal $\P$-prime ideal is given by the $A_d$'s. Then, by \cite[Proposition 3.8]{SSglII}, the Krull--Gabriel dimension of $\Mod_\bfA$ is given by the Krull dimension of the space $\Spec_\P(\bfA)$. The result then follows from Corollary~\ref{cor:dim specA}.
\end{proof}

\subsection{The section functors}\label{subsec:section functors}

We now give an alternative description of the generic categories $\Mod_d^\gen$, and use this to describe the right adjoints to the localization functors $\overline{T}_d$.

Let $K_d=\Frac(\abs{A_d})$ be the fraction field of the domain $\abs{A_d}$; this field comes with an induced action of $\P$, and we let $\abs{K_d}$ denote the field itself without the $\P$-action. Suppose $V$ is a $\abs{K_d}$-vector space equipped with a compatible $\P$-representation structure. There is a $\abs{K_d}$-subspace $V^\pol$ consisting of \emph{polynomial elements}, which are elements of $V$ that generate a polynomial $\P$-representation over $\C$. We say $V$ is a $K_d$\emph{-module} if it is spanned by these polynomial elements over $\abs{K_d}$, and we let $\Mod_{K_d}$ denote the category of $K_d$-modules.

We have a functor $\overline{S}_d:\Mod_{K_d}\to\Mod_{A_d}$ mapping a $K_d$-module $M$ to the set $M^\pol$ of polynomial elements, which is naturally an $A_d$-module. Our goal is to show that $\overline{S}_d$ is right adjoint to $\overline{T}_d$. To do this, we first identify $\Mod_{K_d}$ with $\Mod_d^\gen$.

\begin{lemma}\label{lem:K_d mod equiv gen}
    We have an equivalence of categories $\Mod^\gen_d\cong\Mod_{K_d}$.
\end{lemma}

\begin{proof}
    The functor $\Mod_{A_d}\to\Mod_{K_d}$ given by $M\mapsto K_d\otimes_{A_d} M$ is exact and kills torsion $A_d$-modules. Thus, we have an induced exact functor $F:\Mod_d^\gen\to\Mod_{K_d}$. Notice that $\overline{S}_d$ is right adjoint to $F\overline{T}_d$. Then, the proof that $F$ is an equivalence of categories, with right quasi-inverse $\overline{T}_d\overline{S}_d$, follows the proofs of \cite[Propositions 2.7 and 2.8]{NSSnoethdeg2}.
    \end{proof}

For the remainder of the paper, we identify $\Mod_d^\gen\cong\Mod_{K_d}$, and call $\overline{S}_d$ the \emph{section functor}; it is right adjoint to the localization functor $\overline{T}_d:\Mod_{A_d}\to\Mod_d^\gen$. The following result will be important in \S\ref{sec:ModA} to study the bounded derived category of $\bfA$-modules.

\begin{proposition}\label{prop:saturated objects}
    For any finite length polynomial $\P$-representation $M$, the natural map $A_d\otimes M\to \overline{S}_d(\overline{T}_d(A_d\otimes M))$ is an isomorphism.
\end{proposition}

\begin{proof}
    We show that if $m\in K_d\otimes M$ is a polynomial element, then $m\in A_d\otimes M$. We can write $m$ as
    \[m=\sum_{i=1}^p\frac{f_i}{g}\otimes m_i,\]
    where $f_i,g\in A_d$ with $\gcd(g,f_1,\ldots,f_p)=1$, and $m_i\in M$ linear independent over $\C$. Let $a\gg 0$ be such that $g,f_1,\ldots,f_p\in \mathscr{S}_a(A_d)$. Let $L_{a+1}=\prod_{i=1}^n\GL_{a+1}$ denote the Levi subgroup of $P_{a+1}$. We have that $m$ generates a finite-dimensional polynomial representation of $L_{a+1}$, so let $\{e_1,\ldots,e_r\}$ be a basis, where
    \[e_j=\sum_{k=1}^{n_j}\frac{f_{jk}}{g_j}\otimes m_{jk},\qquad f_{jk},g_j\in A_d,m_{jk}\in M.\]
    Thus, every element of this $L_{a+1}$-representation can be written with common denominator $g'=g_1\cdots g_r$. The hypersurface defined by $g'$ contains the $L_{a+1}$-orbit of the hypersurface defined by $g$, and so this orbit is finite; otherwise, the orbit would be a dense open subset of $\Spec(\mathscr{S}_a(A_d))$ and could not be contained in the hypersurface defined by $g'$. Furthermore, $L_{a+1}$ fixes each irreducible component of the hypersurface $g=0$ as $L_{a+1}$ is connected and $1\in L_{a+1}$.
    The irreducible components are therefore fixed under the $L_{a+1}$-action, and so $L_{a+1}$ acts by a scalar in $\C$ on $g$. In particular, the polynomial $L_{a+1}$-representation generated by $g$ is one-dimensional. Then, the action of $L_{a+1}$ on $g$ is given by products of powers of determinants: if $A\in L_{a+1}$ is a block diagonal matrix given by $(A_1,\ldots,A_n)$ where each $A_i$ is of size $(a+1)\times(a+1)$, then
    \[A\cdot g=\prod_{i=1}^n(\det A_i)^{d_i}g,\qquad d_i\geq 0.\]
    Recall that $g\in \mathscr{S}_a(A_d)$, and is thus invariant under the action of matrices of $L_{a+1}$ where each block along the diagonal is of the form
    \[\begin{pmatrix} \id_a & 0 \\ 0 & \alpha\end{pmatrix},\qquad \alpha\in\C.\]
    Since $g$ is nonzero, this implies that $L_{a+1}$ must actually act trivially on $g$, and so $\bfL$ also acts trivially on $g$. Thus, $g\in \C$, as $(A_d)^\bfL=\C$, and so $m\in A_d\otimes M$, as desired.
\end{proof}

\section{Structure of $\bfA$-modules}\label{sec:ModA}

\subsection{Set-up and notation}

For $d=0,\ldots,n$ let $\Mod_{\bfA,\leq d}$ denote the full subcategory of $\Mod_\bfA$ spanned by $\bfA$-modules locally annihilated by a power of $\frakp_d$, i.e., supported on $V(\frakp_{d})$. This gives a chain of Serre subcategories
\[\Mod_{\bfA,\leq 0}\subset\Mod_{\bfA,\leq 1}\subset\cdots\subset\Mod_{\bfA,\leq n}=\Mod_\bfA.\]
We have Serre quotient categories
\[\Mod_{\bfA,> d}=\Mod_\bfA/\Mod_{\bfA,\leq d},\qquad\Mod_{\bfA,d}=\Mod_{\bfA,\leq d}/\Mod_{\bfA,\leq d-1}.\]
We let $\Mod_{\bfA,d}^\fg$ denote the subcategory of $\Mod_{\bfA,d}$ on finitely generated objects. 

For each $d=0,\ldots,n$, we have the following functors associated to these categories:
\begin{itemize}
    \item The localization functor $T_{>d}:\Mod_\bfA\to\Mod_{\bfA,> d}$.
    \item The localization functor $T_d:\Mod_{\bfA,\leq d}\to \Mod_{\bfA,d}$; this is the restriction of $T_{>d-1}$.
    \item The section functor $S_{>d}:\Mod_{\bfA,> d}\to\Mod_{\bfA}$, which is right adjoint to $T_{>d}$.
    \item The section functor $S_d:\Mod_{\bfA,d}\to\Mod_{\bfA,\leq d}$, which is right adjoint to $T_d$; by uniqueness of right adjoint, it agrees with the restriction of $S_{>d-1}$.
    \item The left-exact saturation functor $\Sigma_{>d}=S_{>d}\circ T_{>d}$. 
    \item The left-exact torsion functor $\Gamma_{\leq d}:\Mod_\bfA\to\Mod_{\bfA,\leq d}$, defined by taking the maximal submodule supported on $V(\frakp_{d})$.
\end{itemize}
For each $d=0,\ldots,n$, recall the following notation from the previous section:
\begin{itemize}
    \item The quotient ring $A_d=\bfA/\frakp_d$ and the category $\Mod_{A_d}$ of $A_d$-modules.
    \item The quotient category $\Mod_d^\gen$ of $\Mod_{A_d}$ by the subcategory of torsion $A_d$-modules, with subcategory $(\Mod_d^\gen)^\fg$ on finitely generated objects (let $\Mod_0^\gen=\Mod_{A_0}$).
    \item The localization functor $\overline{T}_d:\Mod_{A_d}\to\Mod^\gen_d$ (for $d=0$, this is the identity functor).
    \item The section functor $\overline{S}_d:\Mod_d^\gen\to\Mod_{A_d}$ (for $d=0$, this is the identity functor).
\end{itemize}

Let $\thrmD(\bfA)$ denote the derived category of $\bfA$-modules. We have the following associated triangulated categories:
\begin{itemize}
    \item The bounded derived category $\thrmD^b_\fg(\bfA)$ of objects with finitely generated cohomology.
    \item The full subcategory $\thrmD(\bfA)_{\leq d}$ on objects $M$ for which $\thrmR\Sigma_{>d}(M)=0$.
    \item The full subcategory $\thrmD(\bfA)_{>d}$ on objects $M$ for which $\thrmR\Gamma_{\leq d}(M)=0$.
    \item The category $\thrmD(\bfA)_d=\thrmD(\bfA)_{\leq d}\cap\thrmD(\bfA)_{>d}$.
\end{itemize}

\subsection{Decomposition of $\Mod_\bfA$}

The filtration of $\Mod_\bfA$ by the subcategories $\Mod_{\bfA,\leq d}$ provides a way to decompose $\Mod_\bfA$ via the results in \cite[\S4.2]{SSglII}. The key hypothesis needed for these results is the \emph{(Inj) hypothesis}, which is that injective objects in $\Mod_{\bfA,\leq d}$ remain injective in $\Mod_\bfA$. We now verify that this hypothesis holds.

\begin{lemma}
Injective objects in $\Mod_{\bfA,\leq d}$ remain injective in $\Mod_{\bfA}$.
\end{lemma}

\begin{proof}
By \cite[Corollary 4.19]{SSglII}, it suffices to show that the Artin--Rees lemma holds for each $\P$-prime ideal $\frakp_d$ ideal, i.e., that if $M$ is a finitely generated $\bfA$-module and $N\subset M$ is a submodule, then there exists $r$ such that $\frakp_d^m M\cap N=\frakp_d^{m-r}(\frakp_d^r M\cap N)$ for all $m\geq r$.

    Let $a>\ell(M)$. By the classical Artin--Rees lemma, there exists an $r$ such that for all $m\geq r$, we have the desired equality after applying the specialization functor $\mathscr{S}_a(-)$:
    \[\mathscr{S}_a(\frakp_d^mM\cap N)=\mathscr{S}_a(\frakp_d^{m-r}(\frakp_d^rM\cap N)),\qquad m\geq r.\]
    Then, Artin--Rees holds before specialization as well using the same $r$. This is because $a>\ell(M)\geq \ell(\frakp_d^m M\cap N)$, and so if there were a strict containment $\frakp_d^m M\cap N\supsetneq\frakp_d^{m-r}(\frakp_d^rM\cap N)$ for some $m\geq r$, then one would have strict containment after specialization as well.
\end{proof}

This lemma implies the results in \cite[\S4.1-2]{SSglII}. In particular, we have the following.

\begin{lemma}\label{lem:der sat}
    An $\bfA$-module $M$ is derived saturated, i.e., the natural map $M\to\thrmR\Sigma_{>d}(M)$ is an isomorphism, if and only if $\thrmR\Hom_\bfA(N,M)=0$ for all $N\in\Mod_{\bfA,\leq d}$.
\end{lemma}

\begin{lemma}\label{lem:inj bij}
The functors $S_{>d}$ and $T_{>d}$ induce bijections among injectives of $\Mod_{\bfA,>d}$ and injectives $I$ of $\Mod_{\bfA}$ for which $\Gamma_{\leq d}(I)=0$.
\end{lemma}

\begin{remark}\label{rmk:inj A_d}
    The analogous results for $\Mod_{A_d}$ and the subcategory $\Mod_d^\tors$ hold, i.e., $\Mod_d^\tors\subset\Mod_{A_d}$ satisfy the (Inj) hypothesis, and so the results in \cite{SSglII} also apply.
\end{remark}

\subsection{Comparison with generic $\bfA/\frakp_d$-modules}

Recall the categories $\Mod_{\bfA/\frakp_d}=\Mod_{A_{d}}$ of $\bfA$-modules annihilated by $\frakp_d$, and the generic categories $\Mod^\gen_{\bfA/\frakp_d}=\Mod^\gen_d$ studied in \S\ref{sec:gen}. The following results are essentially the same as \cite[Propositions 6.1 and 6.2]{SSglII}, but we provide some added detail and use our notation.

\begin{lemma}\label{lem:gen A_d subcat}
    $\Mod_d^\gen$ is equivalent to the subcategory $\Mod_{A_d,d}\subset\Mod_{\bfA,d}$ on objects of the form $T_{d}(M)$, where $M$ is an $\bfA$-module such that $\frakp_dM$ is supported on $V(\frakp_{d-1})$.
\end{lemma}

\begin{proof}
    The functor $T_{d}:\Mod_{\bfA,\leq d}\to\Mod_{\bfA,\leq d-1}$ is exact. If $M\in\Mod_{A_d}\subset\Mod_{\bfA,\leq d}$, then $\frakp_d M=0$ and so it is supported on $V(\frakp_{d-1})$; thus, $T_d(M)\in\Mod_{A_d,d}$. Furthermore, $T_{d}$ kills objects of $\Mod_{A_d}$ supported on $V(\frakp_{d-1})$, i.e., $\Mod_d^\tors$. Thus, $T_{d}$ induces a functor $\mathscr{F}:\Mod^\gen_d\to \Mod_{A_d,d}$, which we will show is an equivalence of categories.

    We first show that $\mathscr{F}$ is essentially surjective by showing that the localization functor $T_{d}$ restricted to $\Mod_{A_d}$ is essentially surjective. Consider an object of the form $T_{d}(M)\in\Mod_{A_d,d}$, where $M\in\Mod_\bfA$ such that $\frakp_d M$ is supported on $V(\frakp_{d-1})$. Then we have the $A_d$-module $\overline{M}=M/\frakp_d M$, and $M\to\overline{M}$ is a surjection with kernel $\frakp_d M$ supported on $V(\frakp_{d-1})$. Thus, $T_{d}(\overline{M})\cong T_{d}(M)$ as objects in $\Mod_{\bfA,d}$, and so $\mathscr{F}$ is essentially surjective.

    Suppose $M,N\in\Mod_{A_d}$, and recall that $\overline{T}_d:\Mod_{A_d}\to\Mod^\gen_d$ denotes the localization functor. We have that
    \[\Hom_{\Mod_d^\gen}(\overline{T}_d(M),\overline{T}_d(N))\cong \Hom_{\Mod_{A_d,d}}(T_{d}(M),T_{d}(N)),\]
    as both are given by $\varinjlim\Hom_{A_d}(M',N')$, where the colimit is over $A_d$-modules $M',N'$ such that $M'\subset M$ is an inclusion, $N\to N'$ is a surjection, and both $M/M'$ and $\ker(N\to N')$ are in $\Mod_{d}^\tors$. Since $\mathscr{F}(\overline{T}_d(M))=T_{d}(M)$, this shows that $\mathscr{F}$ is fully faithful.
\end{proof}

Recall the section functors $\overline{S}_{d}:\Mod_d^\gen\to\Mod_{A_d}$ studied in \S\ref{subsec:section functors}, and let $\thrmR\overline{S}_{d}$ denote its right derived functor. We now compare $\thrmR\overline{S}_{d}$ and $\thrmR S_{d}$ as functors from the derived category of $\Mod_d^\gen\cong\Mod_{A_d,d}$. The former is computed by considering injective resolutions in $\Mod_d^\gen$, while the latter is computed by considering injective resolutions in $\Mod_{\bfA,d}$. Although injective objects in these categories are different, the following result says that the derived functors still agree; this is key for describing the derived category $\thrmD^b_\fg(\bfA)$.

\begin{proposition}\label{prop:der sec agree}
    Let $M\in\Mod_d^\gen$. Then $\thrmR\overline{S}_{d}(M)\cong\thrmR S_{d}(M)$ are isomorphic.
\end{proposition}

\begin{proof}
    First, by Lemma~\ref{lem:gen A_d subcat}, we may identify the functor $\overline{T}_d:\Mod_{A_d}\to\Mod_d^\gen$ with the restriction of the functor $T_d:\Mod_{\bfA,\leq d}\to\Mod_{\bfA,d}$. Then by uniqueness of right adjoint, we have that $\overline{S}_d$ agrees with $S_d$ on $\Mod_d^\gen$.

    Now, to prove the proposition, we explain why it suffices to show the following claim: if $I\in\Mod_d^\gen$ is injective, then the map $\overline{S}_d(I)\to\thrmR\Sigma_{>d-1}(\overline{S}_d(I))$ is an isomorphism. If so, then since
    \[\thrmR\Sigma_{>d-1}(\overline{S}_d(I))=\thrmR S_{>d-1}(T_{>d-1}\overline{S}_d(I))=\thrmR S_{>d-1}(I)=\thrmR S_d(I),\]
    we have that $\overline{S}_d(I)\cong\thrmR S_d(I)$ and so $I$ is $S_d$-acyclic. Furthermore, $S_d(I)=\overline{S}_d(I)$ by the previous paragraph.  
    Then, suppose $M\to I^\bullet$ is an injective resolution of $M$ in $\Mod_d^\gen$. Since each injective object is $S_d$-acyclic, we see that $S_d(I^\bullet)$ gives $\thrmR S_d(M)$. We also have that $\overline{S}_d(I^\bullet)=S_d(I^\bullet)$, and the former computes $\thrmR\overline{S}_d(M)$. Thus, $\thrmR S_d(M)=\thrmR \overline{S}_d(M)$ agree.

    Let $J=\overline{S}_d(I)\in\Mod_{A_d}$.
    To prove the claim, it suffices to show by Lemma~\ref{lem:der sat} that $\Ext^i_\bfA(N,J)=0$ for all $N\in\Mod_{\bfA,\leq d-1}$ and all $i\geq 0$. First let $i=0$. We may assume that $N$ is finitely generated, and so since $N\in\Mod_{\bfA,\leq d-1}$, it is annihilated by $\frakp_{d-1}^k$ for some $k$. Then by d\'evissage, we may also assume $N$ itself is annihilated by $\frakp_{d-1}$: we have a filtration 
    \[0=\frakp_{d-1}^kN\subset\frakp_{d-1}^{k-1}N\subset\cdots\subset\frakp_{d-1}N\subset N\]
    where each associated graded piece is annihilated by $\frakp_{d-1}$, and if there are no maps from any given piece of the filtration to $J$, then $\Hom_\bfA(N,J)=0$ as well. If $\frakp_{d-1}N=0$, then $\frakp_dN=0$ as well, and so $N$ is a torsion $A_d$-module. But $J=\overline{S}_d(I)$ is saturated with respect to this category, and so $\Hom_\bfA(N,J)=\Hom_{A_d}(N,J)=0$ by Lemma~\ref{lem:der sat}.

    Now suppose $j\geq 1$. By derived adjunction, we have
    \[\thrmR \Hom_{A_d}(N\tensor[\bfA,der]A_d,J)=\thrmR\Hom_\bfA(N,\thrmR\Hom_\bfA(A_d,J))=\thrmR\Hom_\bfA(N,J),\]
    where the last equality is because $J=\overline{S}_d(I)$ is an $A_d$-module. Since $I$ is injective, we have that $J$ is injective as an $A_d$-module by Lemma~\ref{lem:inj bij}, and so we have that
    \[\Ext^i_\bfA(N,J)=\Hom_{A_d}(\Tor^\bfA_i(N,A_d),J).\]
    Note that $\Tor_i^\bfA(N,A_d)$ is supported on $V(\frakp_{d-1})$ since $N$ is. We thus reduce to the $i=0$ case above and see that $\Ext^i_\bfA(N,J)=0$, as desired.
    \end{proof}

\begin{corollary}
    Suppose $M$ is an $A_d$-module. Then $\thrmR^i\Sigma_{>d-1}(M)$ is annihilated by $\frakp_d$ for all $i\geq 0$.
\end{corollary}

\begin{proof}
    We have that $\thrmR^i\Sigma_{>d-1}(M)=\thrmR^iS_d(T_d(M))$. Since $T_d(M)\in\Mod_d^\gen$, Proposition~\ref{prop:der sec agree} shows that $\thrmR^iS_d(T_d(M))=\thrmR^i\overline{S}_d(T_d(M))$, which is an object in $\Mod_{A_d}$.
\end{proof}

\subsection{Generators for $\thrmD^b_\fg(\bfA)$}

In this subsection, we prove the main finiteness results on local cohomology and derived saturation. We first prove the following lemma on finite generation of the derived section functors.

\begin{lemma}\label{lem:fg RS}
    If $M\in\Mod^\fg_{\bfA,d}$, then $\thrmR^i S_{d}(M)$ is a finitely generated $\bfA$-module for all $i\geq 0$, and is zero for $i\gg 0$.
\end{lemma}

\begin{proof}
By d\'evissage, we may assume that $M$ is annihilated by $\frakp_d$, and so $M$ is a finitely generated object of $\Mod^\gen_d$. By Corollary~\ref{cor:gen Amod}, there is a finite injective resolution $M\to\overline{T}_d(A_d\otimes V^\bullet)$, with $V^i$ each finite length injective $\P$-representations. Each $\overline{T}_d(A_d\otimes V^i)$ is $\overline{S}_d$-acyclic. Thus, $\thrmR\overline{S}_d(M)$ is given by
\[\overline{S}_d\overline{T}_d(A_d\otimes V^\bullet)=A_d\otimes V^\bullet,\]
where the equality follows from Proposition~\ref{prop:saturated objects}. Since $M\in\Mod^\gen_d$ by assumption, this is also equal to $\thrmR S_d(M)$ by Proposition~\ref{prop:der sec agree}. The resolution $A_d\otimes V^\bullet$ is finite and each object is finitely generated, and so the result follows.
\end{proof}

\begin{theorem}
    For each $d=0,\ldots,n$, if $M\in \thrmD^b_\fg(\bfA)$, then $\thrmR\Sigma_{>d}(M),\thrmR\Gamma_{\leq d}(M)$ are also in $\thrmD^b_\fg(\bfA)$.
\end{theorem}

\begin{proof}
The proof of \cite[Theorem 6.10]{SSglII} applies here.
\end{proof}

The results of \cite[\S4.3]{SSglII} therefore apply to our filtration of $\Mod_\bfA$. In particular, we have the following semi-orthogonal decomposition of $\thrmD^b_\fg(\bfA)$.

\begin{corollary}\label{cor:semiorthog}
    We have a semi-orthogonal decomposition
    \[\thrmD^b_\fg(\bfA)=\langle\thrmD^b_\fg(\bfA)_0,\ldots,\thrmD^b_\fg(\bfA)_n\rangle.\]
\end{corollary}

Finally, we describe a set of generators of $\thrmD^b_\fg(\bfA)$. Recall that $S_{\underline{\lambda}}$ denotes the simple $\P$-representation corresponding to the $n$-tuple of partitions $\underline{\lambda}$.

\begin{corollary}\label{cor:gens derived cat}
    For $d=0,\ldots,n$, the category $\thrmD^b_\fg(\bfA)_d$ is the triangulated subcategory of $\thrmD^b_\fg(\bfA)$ generated by the objects of the form $A_d\otimes S_{\underline{\lambda}}$. The category $\thrmD^b_\fg(\bfA)$ is generated by all such objects, allowing $d$ to vary.
\end{corollary}

\begin{proof}
By \cite[Proposition 4.14]{SSglII}, we have an equivalence of categories $\thrmR S_d:\thrmD^b_\fg(\Mod_{\bfA,d})\to\thrmD^b_\fg(\bfA)_d$. Finitely generated objects of $\Mod_{\bfA,d}$ have finite length filtrations with associated graded pieces given by objects in $(\Mod^\gen_d)^\fg$, and $\Mod_{\bfA,d}$ generates $\thrmD^b_\fg(\Mod_{\bfA,d})$. Thus, the image of $(\Mod^\gen_d)^\fg$ under $\overline{S}_d$ generates $\thrmD^b_\fg(\bfA)_d$.

By Corollary~\ref{cor:gen Amod}, each object of $(\Mod^\gen_d)^\fg$ has a finite resolution by objects of the form $\overline{T}_d(A_d\otimes V)$, where $V$ is a finite length injective object of $\Rep(\P)$. Then, such objects have finite length filtrations with associated graded pieces of the form $\overline{T}_d(A_d\otimes S_{\underline{\lambda}})$ with $S_{\underline{\lambda}}$ a simple $\P$-representation. The result then follows from $\overline{S}_d(\overline{T}_d(A_d\otimes S_{\underline{\lambda}}))=A_d\otimes S_{\underline{\lambda}}$ by Proposition~\ref{prop:saturated objects}.
\end{proof}

\subsection{Grothendieck groups}

The Grothendieck ring $\thrmK(\Rep(\P))$ is given by $\Lambda^{\otimes n}$, where $\Lambda$ is the ring of symmetric functions. One can consider the category $\Mod_\bfA$ as a module over $\Rep(\P)$, and so $\thrmK(\Mod_\bfA)$ is a module over $\Lambda^{\otimes n}$. In this subsection, we describe this module structure.

\begin{lemma}\label{lem:K(D) mod}
    For each $d=0,\ldots,n$, the Grothendieck group $\thrmK(\thrmD^b_\fg(\bfA)_d)$ is a free $\Lambda^{\otimes n}$-module of rank $1$ with generator $[A_d]$.
\end{lemma}

\begin{proof}
    By Corollary~\ref{cor:gens derived cat}, the $\Lambda^{\otimes n}$-linear map
    \[\thrmK(\Rep(\P))\to\thrmK(\thrmD^b_\fg(\bfA)_d):[V]\mapsto[A_d\otimes V]\]
    is an isomorphism.
\end{proof}

\begin{theorem}\label{thm:groth grp}
    We have an isomorphism of $\Lambda^{\otimes n}$-modules
    \[\thrmK(\Mod_\bfA)\cong \bigoplus_{d=0}^n \thrmK(\thrmD^b_\fg(\bfA)_d).\]
    In particular, $\thrmK(\Mod_\bfA)$ is a free $\Lambda^{\otimes n}$-module of rank $n+1$.
\end{theorem}

\begin{proof}
    By \cite[Proposition 4.17]{SSglII}, the above isomorphisms holds as groups. The projection maps $\thrmK(\Mod_\bfA)\to\thrmK(\thrmD^b_\fg(\bfA)_d)$ giving the isomorphism is defined by $\thrmR\Sigma_{>d-1}\circ\thrmR\Gamma_{\leq d}$, and so the projection maps are $\Lambda^{\otimes n}$-linear. Thus, we have the desired isomorphism of $\Lambda^{\otimes n}$-modules. The second statement now follows from Lemma~\ref{lem:K(D) mod}.
\end{proof}

\subsection{Hilbert series}

In this subsection, we define (enhanced) Hilbert series and show that they are rational for finitely generated $\bfA$-modules. This is a generalization of the definitions and results for $\FI$-modules given in \cite[\S 5]{SSglI}.

Suppose $M$ is a polynomial representation of $\bfL\cong\prod_{i=1}^n\GL$. Equivalently, $M$ can be regarded as an $\FB^{\otimes n}$-module, i.e., a sequence $M=(M_{\underline{a}})_{\underline{a}\in\N^n}$ of representations of products of symmetric groups $\frakS_{\underline{a}}$. Under this identification, we have that the first factor of $\FB$ corresponds to $\GL(\bfV_{(n)})$, while the $n^{\text{th}}$ factor corresponds to $\GL(\bfV_{(1)})$; see \cite[\S 4.1]{Yu} for details.

Let $\underline{\lambda}$ be an $n$-tuple of partitions, with $\underline{a}=(\abs{\lambda^i})$. To define the enhanced Hilbert series, we introduce the following notation. 
\begin{itemize}
    \item Let $c_{\underline{\lambda}}$ denote the conjugacy class of elements of $\frakS_{\underline{a}}$ with cycle type $\underline{\lambda}$.
    \item Let $\Tr(c_{\underline{\lambda}}\mid M)$ denote the trace of the action of $c_{\underline{\lambda}}$ on $M_{\underline{a}}$.
    \item Let $t^{\underline{\lambda}}$ denote the monomial
    \[t^{\underline{\lambda}}=\prod_{i\in[n],\,j\geq 1}t_{ij}^{m_j(\lambda^i)},\]
    where $m_j(\lambda^i)$ denotes the number of times the positive integer $j$ appears in the partition $\lambda^i$.
    \item Let $\underline{\lambda}!$ denote the positive integer
    \[\underline{\lambda}!=\prod_{i\in[n],\,j\geq 1}m_j(\lambda^i)!.\]
\end{itemize}

Suppose $M$ is an $\FB^{\otimes n}$-module with $M_{\underline{a}}$ finite-dimensional for each $\underline{a}$. We define the \emph{(enhanced) Hilbert series} $\widetilde{H}_M(t)$ of $M$ to be the formal series in variables $\{t_{ij}:i\in[n],j\geq 1\}$ given by
\[\widetilde{H}_M(t)=\sum_{\underline{\lambda}}\Tr(c_{\underline{\lambda}}\mid M)\frac{t^{\underline{\lambda}}}{\underline{\lambda}!}.\]
Thus, the isomorphism class of $M$ as an $\FB^{\otimes n}$-module can be completely determined by $\widetilde{H}_M(t)$. Under the equivalence of tensor categories between polynomial representations of $\bfL$ and $\FB^{\otimes n}$, the tensor product of $\FB^{\otimes n}$-modules is defined by Day convolution, which is given as follows. Suppose $M,N$ are $\FB^{\otimes n}$-modules. For an $n$-tuple $\underline{a}=(a_1,\ldots,a_n)$, the $\frakS_{\underline{a}}$-representation corresponding to $M\otimes N$ is given by
\[(M\otimes N)_{\underline{a}}=\bigoplus_{\substack{\underline{b},\underline{c}\in\N^n,\\ a_i=b_i+c_i}}\text{Ind}_{\frakS_{\underline{b}}\times\frakS_{\underline{c}}}^{\frakS_{\underline{a}}}(M_{\underline{b}}\otimes N_{\underline{c}}).\]
Then, the fact that the Hilbert series is multiplicative follows from the $n=1$ case.

For $d\in[n]$, define $T_d\in\Q\llbracket t\rrbracket$ by
\[T_d=\sum_{j\geq 1}(t_{1j}+t_{2j}+\cdots +t_{dj}).\]

\begin{theorem}\label{thm:hilbert series}
    Let $M$ be a finitely generated $\bfA$-module. Then there exist $a\in\N$ and polynomials $p_0(t),p_1(t),\ldots,p_n(t)\in\Q[t_{ij}:i\in[n],j\geq 1]$ so that 
    \[\widetilde{H}_M(t)=p_0(t)+\sum_{d=1}^np_k(t)\exp(T_d).\]
\end{theorem}

\begin{proof}
    We first show that $\widetilde{H}_{A_d}(t)=\exp(T_d)$ for $d\in[n]$. As an $\FB^{\otimes n}$-module, one may observe that the $\frakS_{\underline{a}}$-representation corresponding to $A_d$ is given by
    \[(A_d)_{\underline{a}}=\begin{cases}
        \triv_{a_1}\boxtimes\cdots\boxtimes \triv_{a_d}\quad&\text{$a_i=0$ for $i>d$,}\\
        0\quad&\text{otherwise,}
    \end{cases}\]
    where $\triv_{a}$ denotes the trivial representation of $\frakS_a$. Thus,
    \[\Tr(c_{\underline{\lambda}}\mid A_d)=\begin{cases}
        1\quad&\text{$\lambda^i=0$ for $i>d$},\\
        0\quad&\text{otherwise.}
    \end{cases}\]
    The expression for $\widetilde{H}_{A_d}(t)$ then follows.

    Now, suppose $M= S_{\underline{\lambda}}$ is a simple $\P$-representation. Let $\underline{a}=(\abs{\lambda^i})$. Then as an $\FB^{\otimes n}$-module,
    \[M_{\underline{b}}=\begin{cases}
        M_{\lambda^1}\boxtimes \cdots \boxtimes M_{\lambda^n}\quad&\underline{b}=\underline{a},\\
        0\quad&\text{otherwise},
    \end{cases}\]
    where $M_{\lambda^i}$ denotes the simple representation of $\frakS_{a_i}$ corresponding to the partition $\lambda^i$. It is clear in this case that $\widetilde{H}_M(t)$ is a polynomial.

    The result now follows from Theorem~\ref{thm:groth grp} and the multiplicativity of the Hilbert series.
\end{proof}

\section{Equivalence with $\FI(n)$-modules}\label{sec:fi(n)}

\begin{definition}
    The category $\FI(n)$ of $[n]$-weighted finite sets and injections that do not decrease weights is the following category:
    \begin{itemize}
        \item The objects are $[n]$-weighted finite sets, i.e., finite sets $S=\bigsqcup_{i=1}^n S_i$, where $S_i$ consists of the elements of weight $i$. We also use the notation $S=(S_1,\ldots,S_n)$.
        \item A morphism $S\to T$ is an injection $\varphi:S\to T$ of sets such that weights do not decrease, i.e., for each $i$, we have $\varphi(S_i)\subset\bigsqcup_{j=1}^n T_j$.
    \end{itemize}
\end{definition}

When $n=1$, this category is exactly the classical $\FI$ category of finite sets and injections.

An \emph{$\FI(n)$-module} $M$ is a functor $\FI(n)\to\Vec$, and a morphism of $\FI(n)$-modules is a natural transformation of functors. Let $\Mod_{\FI(n)}$ denote the category of $\FI(n)$-modules.
The goal of this section is to show that $\FI(n)$-modules are equivalent to $\bfA$-modules.

\subsection{Preliminaries on $\FI(n)$-modules}

\subsubsection{Morphisms in $\FI(n)$}

Let $\underline{a}\in\N^n$ be a tuple of nonnegative integers. Similar to the combinatorial categories $\scrC_d$ studied in \S\ref{sec:rep H_d}, objects in $\FI(n)$ can be identified with such tuples; we use $\underline{a}=(a_1,\ldots,a_n)$ to denote the weighted set $([a_1],\ldots,[a_n])$, where the set in the $i^{\text{th}}$ coordinate denotes the elements of weight $i$.

There is a partial order on $\N^n$ known as the \emph{dominance order}: we have that $(a_1,\ldots,a_n)\geq(b_1,\ldots,b_n)$ if and only if
\[a_1+\cdots+a_i\geq b_1+\cdots +b_i,\quad\text{for each $i=1,\ldots,n$}.\]
This partial order describes exactly when a map $\underline{b}\to\underline{a}$ exists in $\FI(n)$. In particular, for a tuple $\underline{a}=(a_1,\ldots,a_n)$, let $\tau(\underline{a})=(a_n,\ldots,a_1)$ denote the reverse tuple. Then there is a map $\underline{b}\to\underline{a}$ if and only if $\tau(\underline{b})\leq\tau(\underline{a})$.

\subsubsection{Simple and projective modules}

The simple $\FI(n)$-objects $\bfM_{\underline{\lambda}}$ are indexed by $n$-tuples of partitions $\underline{\lambda}=(\lambda^1,\ldots,\lambda^n)$ and given by simple representations of products of symmetric groups:
\[\bfM_{\underline{\lambda}}(\underline{b})\mapsto\begin{cases}
    M_{\lambda^1}\boxtimes\cdots\boxtimes M_{\lambda^n}\quad&\text{if $b_i=\abs{\lambda^i}$,}\\
    0\quad&\text{otherwise.}
\end{cases}\]

The main family of $\FI(n)$-modules of interest for us is defined as follows. For $\underline{a}\in\N^n$, let $\scrP_{\underline{a}}$ be the $\FI(n)$-module defined by 
\[\scrP_{\underline{a}}(S)=\C[\Hom_{\FI(n)}(\underline{a},S)].\]
Then, if $M$ is any $\FI(n)$-module, we have by \cite[Proposition 3.2]{SSBrauerI} that
\[\Hom_{\FI(n)}(\scrP_{\underline{a}},M)=M(\underline{a}).\]
Thus, $\scrP_{\underline{a}}$ is a projective $\FI(n)$-module, and we call it the \emph{principal projective} at $\underline{a}$.

The group $\frakS_{\underline{a}}$ acts on $\scrP_{\underline{a}}$ by $\FI(n)$-module automorphisms. If $\underline{\lambda}$ is such that $\abs{\lambda^i}=a_i$ for each $i$, we have a projective $\FI(n)$-module $\scrP_{\underline{\lambda}}$ given by the $\bfM_{\underline{\lambda}}$-isotypic piece of $\scrP_{\underline{a}}$:
\[\scrP_{\underline{\lambda}}=\Hom_{\frakS_{\underline{a}}}(\bfM_{\underline{\lambda}},\scrP_{\underline{a}}).\]
We have that $\scrP_{\underline{\lambda}}$ is the projective cover of $\bfM_{\underline{\lambda}}$. Then, since the $\scrP_{\underline{\lambda}}$'s are direct summands of the principal projectives, all $\FI(n)$-modules have a presentation by principal projective $\FI(n)$-modules.

\subsubsection{Tensor product}

There is a natural symmetric monoidal structure $\amalg$ on $\FI(n)$ given by disjoint union. For two weighted sets $S,T$, their product $S\amalg T$ is $(S_1\sqcup T_1,\ldots,S_n\sqcup T_n)$. In particular, we have that $\underline{a}\amalg\underline{b}=(a_1+b_1,\ldots,a_n+b_n)$.

The monoidal structure on $\FI(n)$ defines a symmetric tensor product $\otimes$ of $\FI(n)$-modules, known as \emph{Day convolution} or \emph{convolution tensor product} (see e.g., \cite[\S 2.1.14]{SSstabpatterns} and \cite[\S 3.10]{SSBrauerI}). This tensor product among principal projectives has an especially nice description: by \cite[Proposition 3.25(d)]{SSBrauerI}, we have that
\[\scrP_{\underline{a}}\otimes\scrP_{\underline{b}}=\scrP_{\underline{a}\amalg\underline{b}}.\]

\subsection{A mapping property}

We describe a mapping property for certain projective objects of $\Rep(\P)$ that will be used to show the equivalence. For a tuple $\underline{a}$, let $W_{\underline{a}}$ denote the principal projective $\P$-representation
\[W_{\underline{a}}=\bigotimes_{i=1}^n (\bfV_{n-i+1})^{\otimes a_i}.\]
It is generated as a $\P$-representation by the vector
\[\epsilon_{\underline{a}}=(e_{n1}\otimes\cdots e_{na_1})\otimes\cdots\otimes (e_{11}\otimes\cdots\otimes e_{1a_n}).\]
Let $\lambda(\underline{a})=((1^{a_n}),\ldots,(1^{a_1}))$ denote the weight of this vector, and for any $\P$-representation $V$, let $V^{\lambda(\underline{a})}$ denote the $\lambda(\underline{a})$-weight space of $V$ as a representation of $\bfL\subset\P$.

\begin{proposition}\label{prop:proj p-mod mapping}
    Let $V$ be a polynomial representation of $\P$. Then the map
    \[\Hom_\P(W_{\underline{a}},V)\to V^{\lambda(\underline{a})}:f\mapsto f(\epsilon_{\underline{a}})\]
    is an isomorphism.
\end{proposition}

\begin{proof}
    If $f:W_{\underline{a}}\to V$ is a map of $\P$-representations, then it must preserve $\bfL$-weight spaces. In addition, $\epsilon_{\underline{a}}$ generates $W_{\underline{a}}$, and so its image under $f$ determines $f$ as a map. Therefore, it suffices to show that the above assignment is surjective.
    
    We first consider the case when $V$ is also a principal projective object $V=W_{\underline{b}}$. Recall from \S\ref{subsec:prelim Rep(P)} that the categories of finite length $\FB(n)$-modules and finite length polynomial $\P$-representations are equivalent.
    Applying this equivalence and the mapping property for principal projectives for $\FB(n)$-modules \cite[\S2.3]{Yu}, the dimension of $\Hom_\P(W_{\underline{a}},W_{\underline{b}})$ as a vector space is equal to that of $\Hom_{\FB(n)}(\underline{b},\underline{a}).$
    
    We now consider the dimension of $(W_{\underline{b}})^{\lambda(\underline{a})}$. If $a_1+\cdots+a_n\neq b_1+\cdots+b_n$, then this weight space is clearly $0$. Otherwise, the dimension agrees with that of $\Hom_{\FB(n)}(\underline{b},\underline{a})$: the tensor factor $(\bfV_{n-i+1})^{\otimes b_i}$ can only contribute to the first $n-i+1$ coordinates of $n$-tuple $\lambda(\underline{a})$, which are $((1^{a_n}),\ldots,(1^{a_i}),0,\ldots,0)$, and so $(W_{\underline{b}})^{\lambda(\underline{a})}$ has a basis indexed by $\Hom_{\FB(n)}(\underline{b},\underline{a})$.
    This shows that $\Hom_\P(W_{\underline{a}},W_{\underline{b}})\cong (W_{\underline{b}})^{\lambda(\underline{a})}$.

    Now suppose $V$ is an arbitrary representation of $\P$; without loss of generality, we may also assume $V$ is of finite length. By \cite[Proposition 4.8]{Yu}, there is a surjection $\varphi:\bigoplus W_{\underline{b}}\to V$ from a finite direct sum of principal projective objects. Let $v\in V^{\lambda(\underline{a})}$, and let $\hat{v}\in\varphi^{-1}(v)$ be a lift of $v$. By the previous paragraph, there exists a $\P$-equivariant map $f:W_{\underline{a}}\to \bigoplus W_{\underline{b}}$ such that
    \[f(\epsilon_{\underline{a}})=\hat{v}\in\bigoplus (W_{\underline{b}})^{\lambda(\underline{a})}.\]
    Then, the composition $\varphi\circ f:W_{\underline{a}}\to V$ gives the desired map of $\P$-representations that shows the above assignment is an isomorphism.
\end{proof}

\subsection{Proof of the equivalence}

We now prove that $\Mod_{\FI(n)}\cong\Mod_\bfA$ as tensor categories. Since both $\Mod_\bfA$ and $\Mod_{\FI(n)}$ are abelian categories with enough projectives, it suffices to show that their respective categories of projective objects are equivalent as tensor categories. 

The projective objects of $\Mod_\bfA$ are exactly those of the form $\bfA\otimes V$, where $V$ is a projective polynomial representation of $\P$. Recall from \S\ref{subsec:prelim Rep(P)} that the projective covers of the simple $\P$-representations are of the form $\bigotimes_{i=1}^n\bfS_{\lambda^i}(\bfV_{n-i+1})$. These representations can all be realized as direct summands of $W_{\underline{a}}$'s. 

For a tuple $\underline{a}\in\N^n$, let $Q_{\underline{a}}$ denote the projective $\bfA$-module $\bfA\otimes W_{\underline{a}}$. It follows that every $\bfA$-module has a presentation by $Q_{\underline{a}}$'s, and so we call these the \emph{principal projective} $\bfA$-modules. To show that the categories of projective objects are equivalent, it is enough to show that the abelian tensor categories generated by the $Q_{\underline{a}}$'s and the principal projective $\FI(n)$-modules $\scrP_{\underline{a}}$ are equivalent.

Note that $Q_{\underline{a}}$ is generated as an $\bfA$-module by the element
\[v_{\underline{a}}=1\otimes \epsilon_{\underline{a}}=1\otimes (e_{n1}\otimes\cdots\otimes e_{n a_1})\otimes\cdots\otimes (e_{11}\otimes\cdots\otimes e_{1a_n}),\]
which has weight $\lambda(\underline{a})=((1^{a_n}),\ldots,(1^{a_1}))$ under the $\bfL$-action. Furthermore, the tensor product $Q_{\underline{a}}\otimes Q_{\underline{b}}$ is also a principal projective. Let $\underline{a+b}$ denote the tuple $(a_1+b_1,\ldots,a_n+b_n)$. Then, we have
\[Q_{\underline{a}}\otimes Q_{\underline{b}}=\bfA\otimes\bigotimes_{i=1}^n(\bfV_{n-i+1})^{\otimes(a_i+b_i)}=Q_{\underline{a+b}}.\]

We show that the assignment $\scrP_{\underline{a}}\mapsto Q_{\underline{a}}$ induces a tensor functor that gives the equivalence. We now define the functor on morphisms. A map $\scrP_{\underline{a}}\to\scrP_{\underline{b}}$ is the same as giving a morphism $\sigma:\underline{b}\to\underline{a}$ in $\FI(n)$. The corresponding map of $\bfA$-modules $f_\sigma:Q_{\underline{a}}\to Q_{\underline{b}}$ is given as follows.

Since $v_{\underline{a}}$ generates $Q_{\underline{b}}$, it suffices to describe its image under $f_\sigma$ so that the restriction of $f_\sigma$ to $W_{\underline{a}}\subset\bfA\otimes W_{\underline{a}}$ is $\P$-equivariant. The image $f_\sigma(v_{\underline{a}})$ is a pure tensor: in the tensor factor of $Q_{\underline{b}}$ corresponding to an element $i\in\underline{b}$ of weight $d$, we have the element that is in the tensor factor of $Q_{\underline{a}}$ corresponding to the element $\sigma(i)\in\underline{b}$; in the tensor factor corresponding to $\bfA$, we have the monomial consisting of the product of $x_{ij}$'s, where the indices are coming from those in $v_{\underline{a}}$ that are not given by a tensor factor in the image of $\sigma$. One sees that the image $f_\sigma(v_{\underline{a}})$ is in the $\lambda(\underline{a})$-weight space of $Q_{\underline{b}}$. By Proposition~\ref{prop:proj p-mod mapping}, this gives a well-defined map of $\P$-modules, and therefore also of $\bfA$-modules.

\begin{example}
    Let $n=3$. Let $\underline{a}=(0,1,2)$ and $\underline{b}=(1,1,0)$, and consider the morphism $\sigma:\underline{b}\to\underline{a}$ in $\FI(n)$ given by mapping the element of weight $1$ in $\underline{b}$ to the element of weight $2$ in $\underline{a}$, and mapping the element of weight $2$ in $\underline{b}$ to the second element of weight $3$ in $\underline{a}$. Then the corresponding map $f_\sigma:Q_{\underline{a}}\to Q_{\underline{b}}$ is given by 
    \[1\otimes e_{21}\otimes e_{11}\otimes e_{12}\mapsto x_{11}\otimes e_{21}\otimes e_{12}\in \bfA\otimes \bfV_3\otimes \bfV_2.\]
\end{example}

\begin{proposition}\label{prop:prinproj equiv}
   The assignment $\sigma\mapsto f_\sigma$ satisfies the following properties:
   \begin{enumerate}
       \item It induces an isomorphism of vector spaces $\Hom_{\FI(n)}(\scrP_{\underline{a}},\scrP_{\underline{b}})\cong\Hom_{\bfA}(Q_{\underline{a}},Q_{\underline{b}})$.
       \item For two morphisms $\sigma:\underline{c}\to\underline{b},\pi:\underline{b}\to\underline{a}$ in $\FI(n)$, we have that $f_{\pi\circ\sigma}=f_\sigma\circ f_\pi$.
       \item For two morphisms $\sigma:\underline{c}\to\underline{a},\pi:\underline{d}\to\underline{b}$ in $\FI(n)$, we have that $f_{\sigma\amalg\pi}=f_\sigma\otimes f_\pi$ as maps $Q_{\underline{a+b}}\to Q_{\underline{c+d}}$.
   \end{enumerate}
\end{proposition}

\begin{proof}
(1): We have that $\Hom_{\FI(n)}(\underline{b},\underline{a})=0$ if and only if $\tau(\underline{b})\not\leq\tau(\underline{a})$. Furthermore, we have that there is no weight space of weight $\lambda(\underline{a})=((1^{a_n}),\ldots,(1^{a_1}))$ in $Q_{\underline{b}}$ when $\tau(\underline{b})\not\leq\tau(\underline{a})$. Thus, $\Hom_{\FI(n)}(\scrP_{\underline{a}},\scrP_{\underline{b}})=0$ implies $\Hom_{\bfA}(Q_{\underline{a}},Q_{\underline{b}})=0$.

Now suppose $\tau(\underline{b})\leq\tau(\underline{a})$. Then the weight space of weight $\lambda(\underline{a})$ in $Q_{\underline{b}}$ is exactly the vector space spanned by the $f_\sigma(v_{\underline{a}})$'s as we range over all morphisms $\sigma:\underline{b}\to\underline{a}$ in $\FI(n)$. Since the $f_\sigma$'s are indeed maps of $\bfA$-modules, we have the desired isomorphism of vector spaces.

(2) and (3): These follow from the naturality of the construction of the $f_\sigma$'s, along with the fact for (3) that $Q_{\underline{a}}\otimes Q_{\underline{b}}\cong Q_{\underline{a+b}}$ is given by the monoidal structure in $\FI(n)$.
\end{proof}

Along with the discussion at the beginning of this subsection, the equivalence of categories now follows.

\begin{theorem}\label{thm:main currying}
    We have an equivalence of tensor categories $\Mod_\bfA\cong \Mod_{\FI(n)}$.
\end{theorem}

\bibliographystyle{alpha}
\bibliography{ref}

\begin{bibdiv}
\begin{biblist}

\bib{cohen}{article}{
      author={Cohen, D.~E.},
       title={On the laws of a metabelian variety},
        date={1967},
        ISSN={0021-8693},
     journal={J. Algebra},
      volume={5},
       pages={267\ndash 273},
}

\bib{CEF}{article}{
      author={Church, Thomas},
      author={Ellenberg, Jordan~S.},
      author={Farb, Benson},
       title={F{I}-modules and stability for representations of symmetric
  groups},
        date={2015},
        ISSN={0012-7094,1547-7398},
     journal={Duke Math. J.},
      volume={164},
      number={9},
       pages={1833\ndash 1910},
}

\bib{CF}{article}{
      author={Church, Thomas},
      author={Farb, Benson},
       title={Representation theory and homological stability},
        date={2013},
        ISSN={0001-8708,1090-2082},
     journal={Adv. Math.},
      volume={245},
       pages={250\ndash 314},
}

\bib{DE}{article}{
      author={Draisma, Jan},
      author={Eggermont, Rob~H.},
       title={Finiteness results for {A}belian tree models},
        date={2015},
        ISSN={1435-9855,1435-9863},
     journal={J. Eur. Math. Soc. (JEMS)},
      volume={17},
      number={4},
       pages={711\ndash 738},
}

\bib{DLL}{article}{
      author={Draisma, Jan},
      author={Laso\'{n}, Micha\l},
      author={Leykin, Anton},
       title={Stillman's conjecture via generic initial ideals},
        date={2019},
        ISSN={0092-7872,1532-4125},
     journal={Comm. Algebra},
      volume={47},
      number={6},
       pages={2384\ndash 2395},
}

\bib{ESS}{article}{
      author={Erman, Daniel},
      author={Sam, Steven~V},
      author={Snowden, Andrew},
       title={Big polynomial rings and {S}tillman's conjecture},
        date={2019},
        ISSN={0020-9910,1432-1297},
     journal={Invent. Math.},
      volume={218},
      number={2},
       pages={413\ndash 439},
}

\bib{Gan}{article}{
      author={Ganapathy, Karthik},
       title={{GL}-equivariance in positive characteristic {I}: the infinite
  exterior algebra},
        date={2024},
        ISSN={1022-1824,1420-9020},
     journal={Selecta Math. (N.S.)},
      volume={30},
      number={4},
       pages={Paper No. 63},
}

\bib{HS}{article}{
      author={Hillar, Christopher~J.},
      author={Sullivant, Seth},
       title={Finite {G}r\"{o}bner bases in infinite dimensional polynomial
  rings and applications},
        date={2012},
        ISSN={0001-8708,1090-2082},
     journal={Adv. Math.},
      volume={229},
      number={1},
       pages={1\ndash 25},
}

\bib{LNNRdim}{article}{
      author={Le, Dinh~Van},
      author={Nagel, Uwe},
      author={Nguyen, Hop~D.},
      author={R\"{o}mer, Tim},
       title={Codimension and projective dimension up to symmetry},
        date={2020},
        ISSN={0025-584X,1522-2616},
     journal={Math. Nachr.},
      volume={293},
      number={2},
       pages={346\ndash 362},
}

\bib{LNNRreg}{article}{
      author={Le, Dinh~Van},
      author={Nagel, Uwe},
      author={Nguyen, Hop~D.},
      author={R\"{o}mer, Tim},
       title={Castelnuovo-{M}umford regularity up to symmetry},
        date={2021},
        ISSN={1073-7928,1687-0247},
     journal={Int. Math. Res. Not. IMRN},
      number={14},
       pages={11010\ndash 11049},
}

\bib{MNP}{article}{
      author={Miller, Jeremy},
      author={Nagpal, Rohit},
      author={Patzt, Peter},
       title={Stability in the high-dimensional cohomology of congruence
  subgroups},
        date={2020},
        ISSN={0010-437X,1570-5846},
     journal={Compos. Math.},
      volume={156},
      number={4},
       pages={822\ndash 861},
}

\bib{MR}{article}{
      author={Murai, Satoshi},
      author={Raicu, Claudiu},
       title={An equivariant {H}ochster's formula for {$\germ{S}_n$}-invariant
  monomial ideals},
        date={2022},
        ISSN={0024-6107,1469-7750},
     journal={J. Lond. Math. Soc. (2)},
      volume={105},
      number={3},
       pages={1974\ndash 2010},
}

\bib{MW}{article}{
      author={Miller, Jeremy},
      author={Wilson, Jennifer C.~H.},
       title={Higher-order representation stability and ordered configuration
  spaces of manifolds},
        date={2019},
        ISSN={1465-3060,1364-0380},
     journal={Geom. Topol.},
      volume={23},
      number={5},
       pages={2519\ndash 2591},
}

\bib{NSsym}{article}{
      author={Nagpal, Rohit},
      author={Snowden, Andrew},
       title={Symmetric ideals of the infinite polynomial ring},
        date={2021},
        note={\arxiv{2107.13027}},
}

\bib{NSSnoethdeg2}{article}{
      author={Nagpal, Rohit},
      author={Sam, Steven~V},
      author={Snowden, Andrew},
       title={Noetherianity of some degree two twisted commutative algebras},
        date={2016},
        ISSN={1022-1824,1420-9020},
     journal={Selecta Math. (N.S.)},
      volume={22},
      number={2},
       pages={913\ndash 937},
}

\bib{SStca}{article}{
      author={Sam, Steven~V},
      author={Snowden, Andrew},
       title={Introduction to twisted commutative algebras},
        date={2012},
        note={\arxiv{1209.5122}},
}

\bib{SSstabpatterns}{article}{
      author={Sam, Steven~V},
      author={Snowden, Andrew},
       title={Stability patterns in representation theory},
        date={2015},
        ISSN={2050-5094},
     journal={Forum Math. Sigma},
      volume={3},
       pages={Paper No. e11, 108},
}

\bib{SSglI}{article}{
      author={Sam, Steven~V},
      author={Snowden, Andrew},
       title={G{L}-equivariant modules over polynomial rings in infinitely many
  variables},
        date={2016},
        ISSN={0002-9947,1088-6850},
     journal={Trans. Amer. Math. Soc.},
      volume={368},
      number={2},
       pages={1097\ndash 1158},
}

\bib{SSglII}{article}{
      author={Sam, Steven~V},
      author={Snowden, Andrew},
       title={G{L}-equivariant modules over polynomial rings in infinitely many
  variables. {II}},
        date={2019},
        ISSN={2050-5094},
     journal={Forum Math. Sigma},
      volume={7},
       pages={Paper No. e5, 71},
}

\bib{SSBrauerI}{article}{
      author={Sam, Steven~V},
      author={Snowden, Andrew},
       title={The representation theory of {B}rauer categories {I}:
  {T}riangular categories},
        date={2022},
        ISSN={0927-2852,1572-9095},
     journal={Appl. Categ. Structures},
      volume={30},
      number={6},
       pages={1203\ndash 1256},
}

\bib{SSspequiv}{article}{
      author={Sam, Steven~V},
      author={Snowden, Andrew},
       title={{Sp}-equivariant modules over polynomial rings in infinitely many
  variables},
        date={2022},
        ISSN={0002-9947,1088-6850},
     journal={Trans. Amer. Math. Soc.},
      volume={375},
      number={3},
       pages={1671\ndash 1701},
}

\bib{Yu}{article}{
      author={Yu, Teresa},
       title={Polynomial functors on flags},
        date={2024},
        note={\arxiv{2402.10648}},
}

\end{biblist}
\end{bibdiv}

\end{document}